\documentclass{birkjour}
\usepackage{amsmath}
\usepackage{hyperref,enumerate}
\usepackage{amssymb}
\usepackage{enumitem}
\usepackage{cancel}
\usepackage{graphicx}
\usepackage[usenames, dvipsnames]{color}
\usepackage{amsthm}
\def\bel{\begin{equation}\label}
\def\eeq{\end{equation}}
\newtheorem{Theorem}{Theorem}[section]
\theoremstyle{Definition}
\newtheorem{Definition}[Theorem]{Definition}
\theoremstyle{Remark}
\newtheorem{Remark}[Theorem]{Remark}

\numberwithin{equation}{section}
\newtheorem{Lemma}[Theorem]{Lemma}

\newtheorem{Proposition}[Theorem]{Proposition}
\newtheorem{Corollary}[Theorem]{Corollary}

\def\ds{\displaystyle}
\def\bel{\begin{equation}\label}
\def\eeq{\end{equation}}
\def\cS{\mathcal{S}}
\def\T{\mathbf C}

\def \p {{p_0}}

\newcommand{\x}{{\mathscr{X}}}
\def \l {{\bar l}}

\def \R {{\mathbb R}}
\def \C {{\bf C}}
\newcommand{\N} {{\mathbb N}}

\def \d {\mathbf d}
\def \ol {\overline}

\def \vsm{\vskip 0.3 truecm}

\usepackage[mathscr]{euscript}

\begin{document}
\title[]{Stabilizability in optimal control}
\author{Anna Chiara Lai}
\address{A. C, Lai, Dipartimento di Scienze di Base e Applicate per l'Ingegneria,
Sapienza Universit\`a di Roma, Via Scarpa 16, 00181, Roma, Italy\\
Telefax (39)(06)442401836,\,\, Telephone (39)(06)  49766555}
\email{anna.lai@sbai.uniroma1.it}
\author{Monica Motta}
\address{M. Motta, Dipartimento di Matematica,
Universit\`a di Padova\\ Via Trieste, 63, Padova  35121, Italy\\
Telefax (39)(49) 827 1499,\,\, Telephone (39)(49) 827 1368}  
\email{motta@math.unipd.it}
\thanks{This research is partially supported by the  Padua University grant SID 2018 ``Controllability, stabilizability and infimum gaps for control systems'', prot. BIRD 187147, and by the Gruppo Nazionale per l' Analisi Matematica, la Probabilit\`a e le loro Applicazioni (GNAMPA) of the Istituto Nazionale di Alta Matematica (INdAM), Italy.}

\begin{abstract}
We extend the classical concepts of sampling and Euler solutions  for control systems associated to discontinuous feedbacks  by considering also the corresponding costs.  In particular, we introduce the notions of Sample and Euler stabilizability to a closed target set  $\C$  {\it with $(\p,W)$-regulated cost},  for some continuous, state-dependent function $W$ and some constant $\p>0$:  it roughly means that we require the existence of  a stabilizing feedback $K$  such that all the corresponding sampling and  Euler solutions starting from a point $z$   have   {\em suitably defined}  finite costs, bounded above by  $W(z)/\p$. Then, we show how the existence of a special, semiconcave  Control Lyapunov Function $W$, called $\p$-Minimum Restraint Function,  allows us to construct explicitly such a feedback $K$.  
When  dynamics and Lagrangian are Lipschitz continuous in the state variable,  we prove that  $K$ as above can be still obtained   if there exists  a $\p$-Minimum Restraint Function  which is  merely Lipschitz continuous.   An example on the stabilizability with $(\p,W)$-regulated cost of the nonholonomic integrator control system associated to {\em any cost} with bounded Lagrangian illustrates the results.  
 \end{abstract}
  
\maketitle

 \section{Introduction}
 In this paper we investigate  in an optimal control perspective the stabilizability to a set  $\T\subset\R^n$   of the nonlinear control system 
\bel{Eintro}
\dot x=f(x,u), \ \ u\in U,  
\eeq 
to which we  associate a
  cost of the form
\bel{minprobintro}
  \int_ 0^{{T}_{x} } l(x(\tau),u(\tau))\, d\tau,  
\eeq
where  $l\ge0$, and  $T_x\le+\infty$ --the {\em exit-time} of $x$-- verifies
\bel{Cintro}
 x(t)  \in \R^n\setminus\C \ \ \text{for all $t\in[0,T_x)$}, \quad \lim_{t\to {T}_{x}^-}  \d (x(t))=0
\eeq
(for any $y\in\R^n$,  $\d(y)$ denotes the  distance of  $y$ from   $\T$). 
In particular, for any $r>0$, we set $B_r(\C):=\{y\in\R^n: \ \d(y)\le r\}$ and assume that:
\vskip 0.2 truecm
\noindent {\bf (H0)} the target set $\T\subset\R^n$ is closed, with compact boundary;  the control set  $U\subseteq\R^m$ is  closed, not necessarily bounded; the functions $f:(\R^n\setminus\C)\times U\to\R^n$, $l:(\R^n\setminus\C)\times U\to[0,+\infty)$ are  uniformly continuous on ${\mathcal K}\times U$   for  any compact set ${\mathcal K}\subset \R^n\setminus\C$;  for any $R>0$ there is some  $M(R)>0$ such that  
			\begin{equation}\label{boundres} 
			|f(x,u)|\le M(R), \quad l(x,u)\le M(R) \qquad\forall(x,u)\in (B_R(\C)\setminus \C)\times U.
			\end{equation}
\vskip 0.2 truecm
\noindent Hence,  for any admissible control, given by a  function $u\in L^\infty_{loc}([0,T_x), U)$\,\footnote{That is, $u:[0,T_x)\to U$ is a measurable function which is essentially bounded on any subinterval $[0,T]\subset[0,T_x)$.},  every Cauchy problem associated to  \eqref{Eintro}   has in general multiple solutions  and  the cost may be finite even if $T_x=+\infty$. 
 
In order to obtain sufficient conditions for the stabilizability  of the system with regulated cost, we consider the Hamiltonian  
\bel{Ham}
\ds H(x,\p,p):=\inf_{u\in U}\big\{\langle p\,,\,f(x,u)\rangle+\p\,l(x,u)\big\},
\eeq 
 and   the following notion, firstly introduced in  \cite{MR13} (in a slightly weaker form).
 \begin{Definition}[$\p$-Minimum Restraint Function]\label{defMRF}
   Let $W:\ol{\R^n\setminus {\T}}\to[0,+\infty)$ be a continuous function, which is  locally semiconcave, positive definite, and  proper on $\R^n\setminus\T$. We say that $W$  is a \emph{$\p$-Minimum Restraint Function -- in short, $\p$-MRF --   for some $\p\ge0$,}  if  it verifies the   {\emph{ decrease condition}}:
    \bel{MRH} 
    H (x,\p, D^*W(x) )\le-\gamma(W(x)) \quad \forall x\in {{\R^n}\setminus\T},  \,\footnote{This  means that $H(x,\p, p )\le-\gamma(W(x))$ for every $p\in D^*W(x)$, where $D^*W(x)$ denotes the set of reachable gradients of $W$ at $x$ (see  Subsection \ref{Sprel}).} 
    \eeq
 for some continuous, strictly increasing function $\gamma:(0,+\infty)\to(0,+\infty)$.
\end{Definition}

 A $\p$-Minimum Restraint Function is at once a  Control Lyapunov Function for  \eqref{Eintro} and, by \cite[Prop. 5.1]{MR13}, also a strict viscosity supersolution to the Hamilton-Jacobi-Bellman equation
	$$
	\sup_{u\in U}\left\{-\langle DW(x)\,,\,f(x,u)\rangle-\p\,l(x,u)\right\}-\gamma(W(x)) =0, \quad x\in\R^n\setminus\C. 
	$$   Hence, when a $\p$--Minimum Restraint Function exists, on the one hand,  the (open-loop) global asymptotic controllability of  the control system \eqref{Eintro}   to $\T$ --namely, that for  any initial condition  $z\in\R^n\setminus\C$  there is an admissible trajectory-control pair  $(x,u)$ to \eqref{Eintro} with $x(0)=z$,  such that $\lim_{t\to {T}_{x}^-}  \d (x(t))=0$ in a certain uniform and stable manner that we will not dwell upon here--  is expected  (see e.g.\cite{Sonn83,SS95,Bacc01}). On the other hand, if $\p>0$,  the existence of an admissible trajectory-control  pair $(x,u)$ with $x(0)=z$ satisfying the cost estimate
\bel{Wbounda}
\int_0^{T_x}l(x(t),u(t))\,dt\le  \frac{W(z)}{\p} 
\eeq	
follows by known optimality principles  (see e.g.  \cite{So99,M04,MS15}).
The main contribution of   \cite{MR13,LMR16}   was to prove  that the existence of a $\p$-Minimum Restraint Function $W$ for some $\p>0$  allows  to produce a pair $(x,u)$ that meets {\em both} of these properties.  
 
A related and important goal  is, given a $\p$-Minimum Restraint Function $W$ for some $\p>0$, to provide a state feedback $K:\R^n\setminus\C\to U$ such that the system $\dot x(t) =f(x(t),K(x(t))$ is {\em globally asymptotically stable to $\C$  and has $(\p,W)$-regulated cost}, that is, such that for any stable trajectory $x$ with $x(0)=z$,   the  cost   $\int_0^{T_x}l(x(t),K(x(t)))\,dt$  is not greater than $W(z)/\p$.  
In this paper  we address the question, left by  \cite{MR13,LMR16} as an open problem, of how to define such a feedback law  through the use of  $W$. In the ideal case in which $W$ is differentiable and there exists a continuous  feedback $K(x)$ such that 
$
 \langle DW(x)\,,\,f(x,K(x))\rangle+\p\,l(x,K(x)) \le-\gamma(W(x))$ for all $x\in\R^n\setminus\C,
$ 
one  easily derives global asymptotic stabilizability with $(\p,W)$-regulated cost. However,  it is a classical matter in nonlinear  control systems that a differentiable Control Lyapunov Function $W$ may not exist and, even if a smooth $W$ exists, a continuous feedback $K$ does not generally exist (see e.g.  \cite{Arts83,Brockett,SS80,Sonn83,SS95,CPT95}).  From here,   the need of considering the nonsmooth version  \eqref{MRH} of the decrease condition and of defining a {\em discontinuous feedback}  $K:\R^n\setminus\C\to U$, which, because of the   unboundedness of $U$,  we can only  assume {\em locally bounded} (see Prop. \ref{3.3} below). In  particular, we   might have $\lim\sup_{x\to\bar x\in\partial\C}|K(x)|=+\infty$.

Identified a feedback $K$ as above, the main issue is   how to interpret  the  discontinuous differential equation and the associated exit-time cost, so that the control system \eqref{Eintro} can be stabilized to the target $\C$ with $(\p,W)$-regulated cost   by  $K$.   For the trajectories of $\dot x(t) =f(x(t),K(x(t))$,  we simply adapt to our setting the nowadays classical notions of sampling and Euler solutions in \cite{CLSS,CLRS}, inspired by  differential games theory  \cite{KS}.   However, our primary objective  is to introduce  {\em a  suitable concept of cost}  associated to a stable sampling or  Euler solution starting from      $z$,  so that such cost  is  bounded above by $W(z)/\p$.  Postponing  the precise definitions to Section \ref{Sdef},  given some constants $r$, $R$ such that $0<r<R$,  we call a sampling pair $(x,u)$  {\em $(r,R)$-stable} when, starting from some $z$ with $r<\d(z)\le R$,  $x$ reaches in a uniform manner the $r$-neighborhood $B_r(\C)$ of the target and, after a time $\bar T_x^r$,   remains  definitely in $B_r(\C)$.  In this case, for all $t\ge 0$ we define the corresponding sampling  cost  as
$$x^0(t):=\int_0^{\min\{t,\bar T_x^r\}}l(x(s),u(s))\,ds$$  and show that $x^0(t)\le W(z)/\p$.  The difficulty in proving  the latter inequality lies in the fact that  $\bar T_x^r$  may  not  be the first instant in which $x$ enters $B_r(\C)$. Consequently, we must estimate the cost in a time interval where  we basically have no information on   $x$, except that it is in $B_r(\C)$ (see Subsection \ref{Subsample} below).  Nevertheless, this is the correct notion of sampling cost. Indeed, let us now define an Euler cost-solution pair $(\x^0,\x)$ 
 as the locally uniform limit on $[0,+\infty)$ of a sequence of sampling cost-trajectory pairs as above,   when the sampling times tend to zero. We  obtain that $\x$  approaches uniformly asymptotically  $\C$,  while $\lim_{t\to T_\x^-}\x^0(t)\le   W(z)/\p$, where  $T_\x\le+\infty$ is the exit-time of $\x$, as in \eqref{Cintro} (see Subsection \ref{3.2}).

Furthermore,  inspired by \cite{R00},  we prove that,  when $f$ and $l$ are locally Lipschitz continuous  in $x$  uniformly w.r.t. $u\in U$   up to the boundary of $\C$,  the  existence of a $\p$-Minimum Restraint Function   $W$ with $\p>0$,  possibly not locally semiconcave but merely locally Lipschitz continuous  on $\overline{\R^n\setminus\C}$, still guarantees stabilizability  of   \eqref{Eintro} to $\C$  with  $(p_0,W)$-regulated cost,  in the sample and Euler sense (see Theorem  \ref{TLip} and Corollary \ref{CLip}).  This result can be  useful for the effective implementation of our feedback construction, as shown by the example in Section  \ref{SecEx}.

  We are  motivated to consider $U$ unbounded and the minimization in the decrease condition  on the whole set $U$ (and not on bounded subsets, as usual) mainly because these assumptions are  met, for instance, in the stabilization of mechanical systems with vibrating controls. These are nonlinear control systems, affine or quadratic in the derivative of the control,   which is considered as an {\em impulsive control} (see \cite{BR10}).  In particular,  the  reparameterized  systems  usually introduced  in the study of  control-polynomial systems satisfy (H0).  As a consequence, as shown in  \cite{LM19},  the results of the present paper are  the starting point for the stabilizability  (with and without  regulated cost) of impulsive control systems.
 In fact, our assumptions include also  the case  where  $U$ is bounded and $f$ and $l$ are continuous on $\R^n\times U$. For $U$ unbounded, they are satisfied, for instance, by the class of control problems in which the input appears inside a saturation nonlinearity, such as
 $
f(x,u)=f_0(x)+\sum_{i=1}^m f_i(x)\sigma_i(u)$, $l(x,u)=l_0(x)+l_1(x)|\sigma_0(u)|$,
 where  $l_0$, $l_1$,  $f_0$, $f_1,\dots, f_m\in C(\R^n)$ and $\sigma_0,\dots,\sigma_m$ are bounded, uniformly continuous maps on $U$. The stabilizability of such control systems  plays a relevant role  both in the literature and in the applications (see  e.g.\cite{BaLin00, Ch01, ChLSo96, SSY94}). 
 
Finally,  the value function 
$$
 V(z):=\inf_{(x,u), \ x(0)=z}\, \int_ 0^{{T}_{x} } l(x(\tau),u(\tau))\, d\tau,
$$
 is clearly bounded above by any $\p$-Minimum Restraint Function divided by $\p$.  Hence 
our approach could be useful to design  approximated optimal closed-loop  strategies, when there exists a sequence of $\p$-Minimum Restraint Functions  approaching $V$, as in \cite{MS15},  or at least to obtain ``safe'' performances, keeping the cost  under the value $W$.   Moreover, when $V\le W/\p$, then $V$ is continuous  on the target’s boundary and this is crucial to establish comparison, uniqueness, and robustness properties for the associated Hamilton–Jacobi–Bellman equation  \cite{M04, MS14, MS15} and to study associated asymptotic and ergodic problems \cite{MS215}.  
From this PDE point of view, problem \eqref{Eintro}--\eqref{minprobintro} has been widely investigated;  a likely incomplete bibliography, also containing applications (for instance, the F\"uller and shape-from-shading problems), includes \cite{IRa95, Ma04, CaSi99, So99}  and the references therein.

 The paper is organized as follows.  In the remaining part of the Introduction  we  provide  some preliminary definitions and notation.  In Section \ref{Sdef} we define precisely the  {\it sample} and {\it Euler  stabilizability} of   \eqref{Eintro} to $\C$  {\it with  $(\p,W)$-regulated cost}, which is shown to be guaranteed by the existence of a      $\p$-Minimum Restraint Function in Theorem \ref{Tintro1}, our main result.  In Section \ref{RMRF} we consider the case of  Lipschitz continuous data, postponing the  proofs  in the Appendix.  An example  on  the stabilizability with regulated cost of the non-holonomic integrator control system concludes the paper (see Section  \ref{SecEx}).

 \subsection{Notation and preliminaries}\label{Sprel}
  For every $r\geq 0$ and $\Omega\subset\R^N$, we set $B_r(\Omega):=\{x\in \R^N\mid d(x,\Omega)\leq r\}$, where $d$ is the usual Euclidean distance. When $\Omega=\{z\}$ for some $z\in\R^N$, we also use of the notation $B(z,r):=B_r(\{z\})$.  For any map  $F:\Omega\to\R^M$  \linebreak  we  call   {\it modulus   (of continuity) of $F$}  any increasing, continuous function  \linebreak $\omega:[0,+\infty)\to [0,+\infty)$ such that $\omega (0)=0$,  $\omega(r)>0$ for every $r>0$ and $|F(x_1)-F(x_2)|\le\omega(|x_1-x_2|)$ for all $x_1$, $x_2\in \Omega$. We   use $\overline{\Omega}$, $\overset{\circ}{\Omega}$ to denote the closure and the interior of  the set $\Omega$, respectively. 
%  As customary, we use the symbol ${\mathcal{KL}}$ to denote the set of all continuous functions 
%$\beta:[0,+\infty)\times[0,+\infty)\to[0,+\infty)$ such that:
% (1)\, $\beta(0,t)=0$ and $\beta(\cdot,t)$ is strictly increasing and unbounded for each $t\ge0$;
% (2)\, $\beta(r,\cdot)$ is decreasing  for each $r\ge0$; (3)\, $\beta(r,t)\to0$ as $t\to+\infty$ for each $r\ge0$. 

Let us summarize  some basic notions in nonsmooth analysis  (see  e.g. \cite{CS}, \cite{CLSW}, \cite{Vinter} for a thorough treatment).
\begin{Definition}[Positive definite and proper functions]  Let  $\Omega\subset \R^N$ be an open set with compact boundary. A continuous function $F:\overline{\Omega} \to\R$ is said {\rm positive definite on $\Omega$} if  $F(x)>0$ \,$\forall x\in\Omega$ and $F(x)=0$ \,$\forall x\in\partial\Omega$. The function $F$ is called {\rm proper  on $\Omega$}  if the pre-image $F^{-1}(K)$ of any compact set $K\subset[0,+\infty)$ is compact.
\end{Definition}

 \begin{Definition} {\rm (Semiconcavity). }\label{sconc} Let $\Omega\subseteq\R^N$. A  continuous function $F:\Omega\to\R$  is said to be {\rm  semiconcave  on $\Omega$} if   there exists $\rho>0$ such that
 $$
 F(x)+F(\hat x)-2F\left(\frac{x+\hat x}{2}\right)\le \rho|x-\hat x|^2,
 $$
 for all   $x$, $\hat x\in \Omega$ such that $[x,\hat x]\subset\Omega$. The constant $\rho$ above is called a {\rm semiconcavity constant} for $F$ in $\Omega$.  $F$ is said to be {\rm  locally semiconcave  on $\Omega$} if it semiconcave on every compact subset of $\Omega$.
 \end{Definition}
\noindent We remind that locally semiconcave functions are locally Lipschitz continuous. Actually, they are twice differentiable almost everywhere.
 \begin{Definition}\label{D*}{\rm (Limiting gradient). }  Let $\Omega\subseteq\R^N$ be an open set, and let  $F:\Omega\to\R$  be a locally Lipschitz continuous function.  For every $x\in \Omega$ we call  {\rm set of limiting gradients} of $F$ at $x$, the set:
$$
D^*{F}(x) := \Big\{ w\in\R^N: \ \  w=\lim_{k}\nabla {F}(x_k), \ \  x_k\in DIFF(F)\setminus\{x\}, \ \ \lim_k x_k=x\Big\},
$$
where   $\nabla$ denotes the classical gradient operator and $DIFF(F)$ is the set of differentiability points of $F$.   
\end{Definition}
\noindent The set-valued map $x\leadsto D^*F(x)$ is upper semicontinuous on $\Omega$,  with nonempty, compact, and not necessarily convex  values.  When $F$ is a locally semiconcave function,  $D^*{F}$ coincides  with the {\em limiting subdifferential} $\partial_LF$, namely,
  $$
  D^*F(x)=\partial_LF(x) := \Big\{\lim \,  p_i: \  p_i\in \partial_PF(x_i), \ \lim\,  x_i=x\Big\} \quad \forall x\in\Omega.
  $$
As usual,   for every  $x\in\Omega$, the  {\em proximal subdifferential} of $F$ at $x$ is given by
$$\begin{array}{l} \partial_PF(x):= \Big\{p \in \R^N\,:\, \exists\, \sigma,\eta>0\, \text{ s.t.,}   \ \forall y\in B(x,\eta), \\
\quad\qquad\qquad \qquad  F(y)-F(x)+\sigma |y-x|^2\geq \langle p,y-x\rangle\Big\}.  
\end{array}$$ 
For locally Lipschitz continuous functions,   the {\em Clarke subdifferential}  $\partial_CF(x)$ of $F$ at $x$, can be defined as $\partial_CF(x):=$co$\partial_LF(x)$. 
 Finally,  locally semiconcave functions  enjoy the following properties (see \cite[Propositions 3.3.1, 3.6.2]{CS}).
\begin{Lemma}\label{Lscv} Let $\Omega\subseteq\R^N$ be an open set and let  $F:\Omega \to\R$   be a   locally  semiconcave function. Then for any compact set $\mathcal{K}\subset \Omega$ there exist some positive constants  $L$ and $\rho$ such that, for any $x\in \mathcal{K}$ 
 \footnote{The inequality (\ref{scvintro}) is usually formulated with the proximal superdifferential  $\partial^P F$. However, this does not make a difference here since $\partial^P F=\partial_C F=co D^* F$ as soon as $F$ is locally semiconcave. Hence (\ref{scvintro}) is true in particular for  any $p\in D^*F(x)$. }, 
\bel{scvintro}
\begin{array}{l}
F(\hat x)-F(x)\le  \langle p,\hat x-x \rangle+\rho|\hat x-x|^2, \\ [1.5ex]
|p|\le L    \quad \forall p\in D^*F(x),
\end{array}
\eeq
for any point  $\hat x\in\mathcal{K}$ such that $[x,\hat x]\subset  \mathcal{K}$.   
\end{Lemma}

\section{Sample and Euler stabilizability with regulated cost}\label{Sdef}
 Let us  introduce the notions of sampling and Euler solutions with regulated cost.  
Hypothesis  (H0) is assumed  throughout the whole section.
 
\begin{Definition}[Admissible trajectory-control pairs and costs]\label{adm}
For every point $z\in\R^n\setminus\T$, we will say that $(x,u)$ is an  {\rm   admissible trajectory-control pair from $z$}  for the control system 
 \bel{Egen}
 \dot x=f(x,u), 
 \eeq
  if there exists $T_x\le +\infty$ such that
 $u\in L^\infty_{loc}([0,T_x),U)$ and $x$ is a Carath\'eodory solution   of {\rm\eqref{Egen}} in $[0,T_x)$ corresponding to  $u$,  verifying $x(0)=z$ and 
$$
x([0,T_x))\subset \R^n\backslash \T  \ \ 
 \displaystyle \text{and, if $T_{ x}<+\infty$,} \  \ \lim_{t\to  T^-_{ x}} {\bf d}(x(t))=0.  
$$
 We shall use ${\mathcal{ A}}_f({ z})$ to denote the family of 
admissible trajectory-control pairs $(x,u)$  from $z$ for the control system {\rm\eqref{Egen}}.  Moreover, we will call  {\rm cost  associated to $(x,u)\in {\mathcal{ A}}_f({ z})$}  the function 
$$
x^0(t):=\int_0^{t } l(x(\tau),u(\tau))\,d\tau \ \ \forall t\in[0,T_x).
$$
 If $T_x<+\infty$, we extend continuously  $(x^0,x)$ to $[0,+\infty)$,  by setting  
$$
(x^0,x)(t)=\lim_{\tau\to T_x^-}(x^0,x)(\tau) \qquad \forall t\ge T_x. 
$$
From now on,  we will always consider admissible trajectories and associated costs defined on  $[0,+\infty)$.
 \end{Definition}
Observe that for any admissible trajectory-control pair defined on $[0, T_x)$, when $T_x<+\infty$  the above limit exists by (H0).   In particular, this follows by the compactness of $\partial\C$  and the boundedness of $f$ and $l$ in any bounded  neighborhood of the target.     
\vsm
A {\it partition}  of 
$[0,+\infty)$ is a sequence $\pi=(t^j) $ such that
$t^0=0, \quad t^{j-1}<t^j$ \, $\forall j\ge 1$,  and
   $\lim_{j\to+\infty}t^j=+\infty$.  The value $\text{diam}(\pi):=\sup_{ j\ge 1}(t^{j }-t^{j-1})$ is called the {\em diameter} or the {\em sampling time} of the sequence 
$\pi$. 
 A {\it feedback for \eqref{Egen}}   is defined to be any locally bounded function $K:\R^n\setminus\C\to U$.  
\begin{Definition}[Sampling trajectory 
and sampling cost]
\label{Ssol}   Given a  
feedback  $K:\R^n\setminus\C\to U$, a partition $\pi=(t^j)$ of 
$[0,+\infty)$,  and a  point
$z\in\R^n\setminus\C$, a {\em $\pi$-sampling trajectory}  for  \eqref{Egen} from $z$  associated to $K$  is a continuous function $x$ defined by recursively solving   
$$
\dot x(t)=f(x(t),K(x(t^{j-1})) \qquad t\in[t^{j-1},t^j],~ (x(t)\in\R^n\setminus\C)
$$
from the initial time $t^{j-1}$ up to time 
$$\tau^j:=t^{j-1}\vee\sup\{\tau\in[t^{j-1},t^j ]: \ x \ \text{is defined on } [t^{j-1},\tau)\},
$$
 where $x(t^0)=x(0)=z$. In this case, the   trajectory $x$ is defined on the right-open interval from time zero up to time
$t^-:=\inf\{\tau^j: \ \tau^j<t^j\}$.  Accordingly,  for every $j\ge1$, we set
\bel{olc}
u(t):= K(x(t^{j-1})) \quad \forall t\in[t^{j-1},t^j)\cap[0,t^-).  
\eeq
The pair $(x,u)$ will be called a $\pi$-sampling trajectory-control pair of \eqref{Egen} from $z$ (corresponding to the feedback $K$). 
The {\em sampling cost} associated to $(x,u)$ is given by
\bel{Scost}
x^0(t):=\int_0^t l(x(\tau),u(\tau))\,d\tau \quad t\in[0,t^-).
\eeq
  \end{Definition}

\begin{Definition}[Sample stabilizability with $(p_0,W)$-regulated cost] \label{sstab}    A feedback 
$K:\R^n\setminus\C\to U$ is said to {\rm sample-stabilize    \eqref{Egen} to $\C$}  
 if there is a  function 
$\beta\in{\mathcal {KL}}$  satisfying the following: for each pair $0<r<R$  
 there exists  $\delta=\delta(r,R)>0$,  such that, 
for every partition $\pi$ of 
$[0,+\infty)$ with  $\text{\em diam}(\pi)\le\delta$ and for any initial state 
$z\in\R^n\setminus\C$ such that ${\bf d}(z)\le R$, any  $\pi$-sampling  
trajectory-control pair  $(x,u)$  of   \eqref{Egen} from $z$ associated to $K$ belongs to ${\mathcal A}_f(z)$  and verifies:
\bel{betaS}
{\bf d}(x(t))\le\max\{\beta(R,t), r\} \qquad \forall t\in[0,+\infty).
\eeq
Such $(x,u)$ are called \emph{$(r,R)$-stable (to $\C$) sampling trajectory-control pairs}. If the system  \eqref{Egen} admits a sample-stabilizing feedback to $\C$, then it is called \emph{sample stabilizable (to $\C$)}.  

When there exist $\p>0$ and a  continuous map $W:\ol{\R^n\setminus\T}\to[0,+\infty)$ which is  positive definite and proper on $\R^n\setminus\C$, such that  the sampling cost $x^0$ associated to any $(r,R)$-stable sampling  pair $(x,u)$ verifies
\bel{estWc}
 x^0(\bar T_x^r)=
\int_0^{\bar T_x^r} l(x(\tau),u(\tau))\,d\tau\le  \frac{W(z)}{\p} 
\eeq
where 
\bel{bar T}
 \bar T_x^r:=\inf\{t>0: \ \d(x(\tau))\le r \ \ \forall \tau\ge t\},
 \eeq  
   we say that  \eqref{Egen} is   \emph{ sample stabilizable (to $\C$) with $(p_0,W)$-regulated cost.}
\end{Definition}
Observe that, when $\d(z)\le r$,  the time $\bar T_x^r$ may be zero. In this case \eqref{estWc}  imposes no conditions on the cost.
\vsm
Let us now introduce Euler solutions and 
 the associated costs and 
a notion of Euler stabilizability to $\C$ with $(p_0,W)$-regulated cost. 
\begin{Definition}[Euler trajectory
  and Euler cost]\label{EulerS} 
Let $(\pi_i)$ be a sequence of partitions of $[0,+\infty)$ such that $\delta_i:=\text{\em diam}(\pi_i)\to 0$ as $i\to \infty$. Given a feedback 
$K:\R^n\setminus\C\to U$ and $z\in\R^n\setminus\C$, for every $i$, 
let $(x_i,u_i)\in{\mathcal{ A}}_{f}({ z})$ be a  $\pi_i$-sampling trajectory-control pair  of
 \eqref{Egen}  from $z$ associated to $K$ and let $ x_i^0$ be the corresponding cost.    If there exists a map $\x:[0,+\infty)\to  \R^n$, verifying
\begin{align}
  x_i\to \x \qquad \text{locally  uniformly in }[0,+\infty) \label{eul1}
\end{align}
 we call $\x$  an \emph{Euler trajectory  of \eqref{Egen} from $z$} (corresponding to the feedback $K$).
 
If moreover  there is a map $\x^0:[0,+\infty)\to  [0,+\infty)$  verifying
\begin{align}
 x_i^0\to \x^0\qquad \text{locally  uniformly in }[0,+\infty), \label{eul2}
\end{align}
 we call  $\x^0$ the \emph{Euler cost} associated to $\x$.
 \end{Definition}

\begin{Remark}\label{rmkeulercost}{\rm As Euler trajectories are not, in general, classical solutions to the control system  \eqref{Egen},  an  Euler cost   may not coincide with  the integral of the  Lagrangian along the corresponding Euler trajectory,   
for some control.  Nevertheless, this is  true  in  special situations, as, for instance, when the function $l$ is continuous,  bounded and  does not depend on the control, that is $l(x,u)=\tilde l(x)$ for all $(x,u)$. Indeed, in this case  if 
there exists a sequence of sampling trajectories $x_i\to \x$ locally   uniformly in $[0,+\infty)$,  the dominated convergence theorem implies that  the associated costs $x^0_i$   converge locally uniformly to the function $\x^0$ verifying
$$\x^0(t)=  \int_0^t \tilde l(\x(\tau))d\tau \quad \text{for any $t\ge0$.}
$$
Indeed,  fixed $t>0$, for all $s\in[0,t]$  one has
$$
|x_i^0(s)-\x^0(s)|\le \int_0^s|\tilde l(x_i(\tau))-\tilde l(\x(\tau))|\,d\tau \le t\omega(\sup_{\tau\in[0,t]}|x_i(\tau)-\x(\tau)|),
$$
when $\omega$ denotes a modulus of $\tilde l$ on a suitable compact neighborhood of $\x([0,t])$. Therefore, $\sup_{s\in[0,t]} |x_i^0(s)-\x^0(s)|\to0$ as $i\to+\infty$, for every $t>0$.}
\end{Remark}
 
\begin{Definition}[Euler stabilizability with $(p_0,W)$-regulated cost] \label{Estab}
The system \eqref{Egen} is \emph{Euler stabilizable}  to $\C$ with  \emph{Euler stabilizing feedback} $K$, 
if there exists a function $\beta\in{\mathcal {KL}}$ such that for each $z\in 	\R^n\setminus\C$,   every Euler solution $\x $ of \eqref{Egen} from $z$  associated to $K$
verifies
\bel{betaeulerdef}
{\bf d}(\x(t))\le\beta(\d(z),t) \qquad \forall t\in[0,+\infty).
\eeq
When there exist some $\p>0$ and a  continuous map $W:\ol{\R^n\setminus\T}\to[0,+\infty)$ which is positive definite and proper on $\R^n\setminus\T$,  such that every Euler cost $\x^0$ associated to $\x$  verifies
\bel{VSeulerdef}
 \lim_{t\to   T_{\x}^-}\x^0(t)\leq  \frac{W(z)}{\p}  \quad\forall z\in 	\R^n\setminus\C, 
\eeq
where   
$$
\displaystyle T_\x:=\inf  \{\tau\in(0,+\infty]: \  \x([0,\tau))\subset\R^n\setminus\C,   \  \  \lim_{t\to  \tau^-} {\bf d}(\x(t))=0\},
$$
then \eqref{Egen} is said to have a \emph{$(p_0,W)$-regulated cost} (w.r.t.  the feedback $K$). 
 \end{Definition}

%%%%%%%%%%%%%%%%%%%%%%%%%%%%%%%%%%%%%%%%%%%%%%%%%%%%%%%%%%%%%%%%%%%%%%%%%%%%
  
\section{Main result}\label{SM1}
\begin{Theorem}\label{Tintro1}
Assume hypothesis {\rm (H0)} and let $W$ be a $\p$-MRF with $\p>0$. Then there exists a locally bounded feedback $K:\R^n\setminus\C\to U$  that sample and Euler stabilizes  system \eqref{Egen}   
 to $\C$ with $(p_0,W)$-regulated cost. 
\end{Theorem}
We split the proof of Theorem \ref{Tintro1} in two subsections,  concerning with the sample stabilizability  and the Euler stabilizability, respectively.
 
 \vsm
Preliminarily,  let us observe that   for any $(x, p)\in(\R^n\setminus\C)\times\R^n$ the infimum in the definition of the Hamiltonian  $ H$ can be taken over a compact subset  of $U$, in view of the following result.
\begin{Proposition}\label{3.3} 
Assume {\rm (H0)} and   let $W$ be a $\p$-MRF with $\p\ge0$.  Then  there exists a continuous  function $N: (0,+\infty)\to (0,+\infty)$ such that,  setting for any  $(x, p)\in(\R^n\setminus\C)\times\R^n$, 
 $$
 H_{N(r)}(x,\p,p):=\min_{u\in U\cap B(0,N(r))}\Big\{\langle p, f(x,u)\rangle+\p\, l(x,u) \Big\} \quad \forall r>0,
 $$
one has
\bel{c2'}
 H_{ N(W(x))} (x,p_0,D^*W(x))< -\gamma(W(x)) \qquad \forall x\in \R^n\setminus\C.
\eeq
 \end{Proposition}
 \begin{proof} 
Fix $\sigma>0$.  By \cite[Prop. 3.3]{LMR16}  we  derive that  there exists a decreasing, continuous function $N: (0,\sigma] \to (0,+\infty)$ such that,  setting
\bel{hL}
 H_{  N(r)}(x,p_0,p):=\min_{u\in U\cap B(0,N(r))}\Big\{\langle p, f(x,u)\rangle+\p\,l(x,u) \Big\}  
\eeq
for all $r\in(0,\sigma]$,  it follows that
\bel{c2'}
 H_{ N(W(x))} (x,\p,D^*W(x))< -\gamma(W(x)) 
\eeq
for every $x\in W^{-1}((0,\sigma])$.  It only remains to show that there exists a   continuous  map $N:[\sigma,+\infty)\to (0,+\infty)$  such that 
extending \eqref{hL}  to $r\in[\sigma,+\infty)$ one gets \eqref{c2'} for every $x\in W^{-1}([\sigma,+\infty))$.
Arguing as in the proof of  \cite[Prop. 3.3]{LMR16}, one can obtain that for any $r>\sigma$  there is some $N(r)\ge N(\sigma)$ such that 
$$
 H_{  N(r)} (x,p_0,p)< -\gamma(W(x)) \qquad\forall x\in W^{-1}([\sigma,r]) \ \text{ and } \  p\in D^*W(x).
 $$
 Moreover, for any $r_2>r_1\ge\sigma$, one clearly has $N(r_2)\ge N(r_1)$ and, enlarging $N$ if necessary, one can assume that $r\mapsto N(r)$ is increasing and  continuous on $[\sigma,+\infty)$.  Therefore for any $x\in W^{-1}([\sigma,+\infty))$ the thesis 
\eqref{c2'} follows  from \eqref{hL} as soon as $r=W(x)$.  \end{proof}

As an immediate consequence  of Proposition \ref{3.3}, the existence of a $\p$-MRF $W$  guarantees the existence of a  feedback $K$ with the following properties.
 \begin{Proposition}\label{FDBK}  Let $W$ be a $\p$-MRF with $\p\ge0$   and fix a selection $p(x)\in D^*W(x)$ for any  $x\in \R^n\setminus\C$. Given a function  $N$ as in   Proposition  \ref{3.3}, then any map  $K:\R^n\setminus\C\to U$ such that
 $$
 K(x)\in \underset{u\in U\cap B(0,N(W(x))}{\arg\min}\Big\{\langle p(x), f(x,u)\rangle+\p\,l(x,u) \Big\},
 $$ 
verifies
\bel{feed1}
\langle p(x), f(x,K(x))\rangle+\p\, l(x,K(x)) 
<-\gamma(W(x)) \quad \forall x\in \R^n\setminus\C.
\eeq
 \end{Proposition}
 We call any map   $K$ as   above, a {\em $W$-feedback} for the control system \eqref{Egen}.
%\bel{S}
%\dot x=f(x,u).
%\eeq 
When the dependence of $K$ on $W$ is clear, we   simply call $K$ a feedback. 
 
\subsection{Proof of the sample stabilizability  with $(p_0,W)$-regulated cost}\label{Subsample}
The proof relies on  Propositions \ref{3.5}, \ref{pinvariant} and on Lemma \ref{LW=d} below.  
 \begin{Proposition}\label{3.5} {\rm \cite[Prop. 3.5]{LMR16}} Assume {\rm (H0)}.
  Let $W$ be a $\p$-MRF with $\p\ge0$,  define $N$ accordingly to
Proposition  \ref{3.3},  and let $K$ be a $W$-feedback. Moreover, let $\varepsilon$,  
$\hat\mu$, $\sigma$  verify $\varepsilon>0$ and $0<\hat\mu< \sigma$.  Then there 
exists some $\hat\delta=\hat\delta(\hat\mu,\sigma)>0$ such that, for every  partition $\pi=(t^j)$ of 
$[0,+\infty)$ with {\em diam}$(\pi)\le\hat\delta$  and  for each $z\in\R^n\setminus\T$ satisfying $W(z)\in(\hat\mu,\sigma]$, any $\pi$-sampling trajectory-control pair  $(x,u)$ of 
\bel{sis1}
\dot x=  f(x,u), \qquad x(0)=z,
\eeq
associated to the feedback $K$ is defined on $[0,\hat t)$   and enjoys the following properties:
\begin{itemize}
\item[(i)]  $\hat t:=T^{\hat\mu}_x<+\infty$, where
\bel{Tr} 
T^{\hat\mu}_x:=\inf\{t\ge 0: \ \ W(x(t))\le \hat\mu\};
\eeq
 \item[(ii)] 
  for every  $t\in[0,\hat t)$ and $j\ge1$   such that  $t\in[t^{j-1}, t^j)$,
\bel{dainserire}
  W(x(t))-W(x(t^{j-1}))+\p  \int_{t^{j-1}}^t  l(x(\tau),u(\tau))\,d\tau  
\le -\frac{\gamma(W(x(t^{j-1}))) }{\varepsilon+1}(t -t^{j-1}).
\eeq
\end{itemize}
 \end{Proposition}
 Proposition \ref{3.5}  describes  the behavior of any sampling trajectory-control pair $(x,u)$  with  sampling time not greater than $\hat\delta$ just until its  first exit-time $\hat t$  from the set $\{x\in \overline{\R^n\setminus \T}: \ \  W(x)>\hat\mu\}$. In \cite{LMR16} this was enough  to derive global asymptotic controllability. Global asymptotic stabilizability, instead,   requires  also that,  loosely speaking,  any  $x$  is defined in $[0,+\infty)$   and stays in the sublevel set $\{x\in \overline{\R^n\setminus \T}: \ \  W(x)\leq\hat\mu\}$ for every $t\ge\bar t$, for some $\bar t=\bar t(\hat\mu,\sigma)$.  This is the content of the next proposition, which  can be seen  as an extension of \cite[Lemma IV.2]{CLSS} to the setting considered here. 

\begin{Proposition}\label{pinvariant} Assume {\rm (H0)} and let $W$ be a $\p$-MRF with $\p\ge0$. 
Using the same notation of Proposition \ref{3.5},  set
\bel{conddelta}
\bar\delta= \bar\delta(\hat\mu,\sigma):=\min\left\{\hat\delta\left(\frac{\hat\mu}{4},2\sigma\right), \frac{\hat \mu}{4L\,m}\right\},
\eeq 
where $L$ is the Lipschitz constant  of $W$ in $W^{-1}([\hat \mu/4,2\sigma])$ and 
\bel{m1}
m:=\sup_{W^{-1}((0,2\sigma])\times U} |(f, l)|.
\eeq
Then for every  partition $\pi=(t^j)$ of 
$[0,+\infty)$ with {\em diam}$(\pi)\le\bar\delta$  and  for each $z\in\R^n\setminus\T$ satisfying $W(z)\in(\hat\mu,\sigma]$, any $\pi$-sampling trajectory $x$ of \eqref{sis1} is defined in $[0,+\infty)$ \footnote{When $T_x<+\infty$, we always mean  that $x:[0,T_x)\to\R^n\setminus\C$ is extended to $[0,+\infty)$ as described in Definition \ref{adm}. 
} and   verifies  
\bel{invariance}
x(t)\in W^{-1}([0,\hat\mu]) \quad \forall t\geq \bar t,
\eeq
where $\bar t:=T_x^{\hat\mu/4}<+\infty$.
\end{Proposition}
\begin{proof}
 Fix a partition $\pi=(t^j)$ of 
$[0,+\infty)$ of diameter not greater than $\bar\delta$ and an initial datum $z\in W^{-1}((\hat \mu, \sigma])$. By Proposition \ref{3.5}  with $\hat\mu/4$ in place of $\hat\mu$,   any $\pi$-sampling solution $x$  is defined at least up to $\bar t:=T_x^{\hat\mu/4}<+\infty$ and  
$W(x([0,\bar t])\subset[\hat \mu/4, W(z)]$,  $W(x(\bar t))=\hat \mu/4$.  Moreover, if $\bar  n:=\max\{j\in\N: \  t^j\leq \bar t\}$, then  we have
\bel{eqbase}
x(t)\in W^{-1}([\hat \mu/4,W(t^{\bar n-1})])\subseteq W^{-1}([\hat \mu/4, 3\hat\mu/4]) \quad \forall t\in [t^{\bar n-1},t^{\bar n}].
\eeq  
where, to deal with  the case $\bar n=0$, we set $t^{-1}:=t^0=0$. The last inclusion follows by the definition of $\bar\delta$, which implies 
$$
W(x(t^{\bar n}))-W(x(\bar t))\le L|x(t^{\bar n})-x(\bar t)|\le Lm\bar\delta\le\frac{\hat \mu}{4}, 
$$
so that $W(x(t^{\bar n}))\le \hat \mu/2$ and, arguing similarly,   $W(x(t^{\bar n-1}))\le 3\hat \mu/4$.

\noindent We use \eqref{eqbase} as base to inductively prove that any  $\pi$-sampling solution  $x$ of \eqref{sis1} either is defined   on $[0,+\infty)$ and verifies \eqref{invariance} in the stronger form
$$
x(t)\in W^{-1}((0,\hat\mu]) \quad \forall t\geq \bar t,
$$
or $x$ has finite blow-up time  coinciding with the first time   $T_x$  such that $\lim_{t\to T_x^-}\d(x(t))=0$:  in this case, since $|\dot x|$ is bounded by $m$,   $x$  can be continuously extended to  $[0,+\infty)$  and this extension verifies \eqref{invariance}.

Fix $j\ge\bar n$ and 
assume by induction that an arbitrary
$\pi$-sampling trajectory $x$, eventually extended accordingly to Definition  \ref{adm},  is defined up to time $t^{j-1}$ and verifies $x([0,t^{j-1}])\subseteq W^{-1}([0,\hat\mu])$.  We have to show that $x$ is defined on  $[t^{j-1},t^{j}]$ and verifies
\bel{induction}
  x(t)\in W^{-1}([0,\hat\mu]) \quad \forall t\in [t^{j-1},t^{j}].
\eeq
If $W(x(t^{j-1}))=0$, $x$ is constant on  $[t^{j-1},t^j]$ and   \eqref{induction} is obviously satisfied. 
When $0<W(x(t^{j-1}))\le\hat\mu$,  we distinguish the following situations:
 
{\sc Case 1.}  $W(x(t^{j-1}))\ge \hat \mu/2$. Then by Proposition \ref{3.5} (choosing in particular $z=x(t^{j-1})$ and the partition $\pi_j:=(t^{k+j-1}-t^{j-1})_k$)  we deduce that any $\pi$-sampling trajectory  with value  $W(x(t^{j-1}))\ge \hat \mu/2$ is defined on the whole interval $[t^{j-1},t^j]$ and verifies $0\le W(x(t))-W(x(t^j))\le Lm\bar\delta\le\hat\mu/4$ for all $t\in [t^{j-1},t^j]$, so that  $x([t^{j-1},t^{j}])\subset W^{-1}([\hat\mu/4,\hat\mu])$ and this implies  \eqref{induction}.
 
{\sc Case  2.}   $W(x(t^{j-1}))<\hat\mu/2$.  Any $\pi$-sampling solution $x$ of \eqref{sis1} with this property can be defined on a maximal interval $[t^{j-1}, \tilde t)$. Assume first that  $\tilde t> t^j$, so that $x$ is defined for all $t\in[t^{j-1},t^j]$ and suppose by contradiction  
$$
x([t^{j-1},t^j])\nsubseteq W^{-1}([0,\hat\mu]).
$$
Then there exist $t^{i-1}<\underbar t^j< \bar t^j\leq t^j$ such that 
$$W(x(\underbar t^j))= \hat \mu/2,\quad  W(x(\bar t^j))=\hat\mu, \quad x([\underbar t^j,\bar t^j])\subseteq W^{-1}([\hat\mu/2,\hat\mu]).$$
This yields the required contradiction, since we have
$$\hat\mu/2=W(x(\bar t^j))-W(x(\underbar t^j))\leq  Lm\bar\delta \le\hat\mu/4.
$$
Therefore $x$ verifies  \eqref{induction}. 

\noindent Let us now assume $\tilde t\le t^j$. By standard properties of the ODEs,  the blow-up time $\tilde t$  verifies either $\lim_{t\to \tilde t^-}|x(t)|=+\infty$ or  $\tilde t=T_x$. Notice that if we had
\bel{liv}
x([t^{j-1},\tilde t))\nsubseteq W^{-1}([0,\hat\mu]),
\eeq
 we could find $t^{i-1}<\underbar t^j< \bar t^j<\tilde t$ and obtain a contradiction arguing as above.
Hence  $x([t^{j-1},\tilde t))\subseteq W^{-1}([0,\hat\mu])$ and $\tilde t= T_x$ necessarily, since the set $W^{-1}([0,\hat\mu])$ is compact.   By the boundedness of $f$ on $W^{-1}((0,\hat\mu])\times U$ this implies that 
 $\exists\lim_{t\to \tilde t}x(t)=\bar z\in\partial\C$ and the extension of $x$ to $[t^{j-1},t^j]$ given by  $x(t)=\bar z$ for all $t\in[\tilde t,t^j]$  verifies  \eqref{induction}. The proof is thus concluded. \end{proof}

\vsm

Finally, let us  relate the level sets of a $\p$-MRF $W$ with the ones of the distance function $\d$ using the following general result.
 \begin{Lemma}\label{LW=d}
 Let $W$, $W_1 :\ol{\R^n\setminus \T}\to[0,+\infty)$ be   continuous functions, and let us assume that $W$ and $W_1$ are  positive definite and  proper on $\R^n\setminus\T$.  Then the functions  $\bar g$, $\underline{g}:(0,+\infty)\to (0,+\infty)$ given by
$$
\begin{array}{l}
\underline{g}(r)=\underline{g}_{\,W,W_1}(r):=\sup\left\{\alpha>0: \ \ \{\tilde z: \ W(\tilde z)\le\alpha\}\subseteq  \{\tilde z: \  W_1(\tilde z) < r\}\right\}, \\ [1.5ex] 
 \bar g(r)=\bar g_{\,W,W_1}(r):=\inf\left\{\alpha>0: \ \ \{\tilde z: \ W(\tilde z)\le\alpha\}\supseteq  \{\tilde z: \  W_1( \tilde z)\le r\}\right\}, 
\end{array}
$$
are well-defined, increasing  and there exist  the limits
\bel{Llim}
\lim_{r\to 0^+}\bar g(r)=\lim_{r\to 0^+}\underline{g}(r)=0, \qquad \lim_{r\to +\infty}\bar g(r)=\lim_{r\to +\infty}\underline{g}(r)=+\infty.
\eeq
Moreover, one has
\bel{Ldis}
\underline{g}(W_1(x))\le W (x)\le \bar g(W_1(x)) \quad \forall   x\in\R^n\setminus\C.
\eeq
\end{Lemma}
\begin{proof}   For every $\alpha>0$, let us introduce the sets $\cS_\alpha:= \{\tilde z: \ W(\tilde z)\le\alpha\}$,   $\cS^1_\alpha:=  \{\tilde z: \ W_1( \tilde z)\le\alpha\}$, and  $\cS^{1_<}_\alpha:=  \{\tilde z: \ W_1( \tilde z)<\alpha\}$. 
By the hypotheses on $W$ and $W_1$ it follows that $(\cS_\alpha)_{\alpha>0}$,   $(\cS^1_\alpha)_{\alpha>0}$, and  $(\cS^{1_<}_\alpha)_{\alpha>0}$ are  strictly  increasing families of nonempty, bounded sets verifying
$$
\begin{array}{l}
\ds\lim_{\alpha\to 0^+}\cS_\alpha= \lim_{\alpha\to 0^+}\cS^1_\alpha=\lim_{\alpha\to 0^+}\cS^{1_<}_\alpha=\C, \\
\ds \lim_{\alpha\to +\infty}\cS_\alpha= \lim_{\alpha\to +\infty}\cS^1_\alpha= \lim_{\alpha\to +\infty}\cS^{1_<}_\alpha=\R^n.
\end{array}
$$
Then for any $r>0$ there exist $\bar \alpha$, $\bar\alpha_1>0$ such that $\cS_\alpha\subset \cS^{1_<}_r$ for all $\alpha\le\bar\alpha$ and   $\cS_\alpha\supset \cS^1_r$ for all $\alpha\ge\bar\alpha_1$, so that  $\bar g(r)$ and $\underline{g}(r)$ turn  out to be well-defined. Moreover, $r\mapsto \bar g(r)$, $\underline{g}(r)$ are clearly increasing and verify the limits \eqref{Llim}.  

 In order to prove the inequalities in  \eqref{Ldis}, given $x\in\R^n\setminus\C$, let us set $r:=W_1(x)$, $\alpha:=W(x)$, $\underline{\alpha}:=\underline{g}(W_1(x))$, and $\bar\alpha:=\bar g(W_1(x))$. Arguing by contradiction, let us first assume that $\underline{g}(W_1(x))> W (x)$, namely
\bel{contr1}
\underline{\alpha} > \alpha. 
\eeq
By the definition of $\underline{\alpha}$, \eqref{contr1} would imply that $\cS_\alpha\subseteq \cS^{1_<}_r$. This is impossible, since $x\in\cS_\alpha$ but $x\notin \cS^{1_<}_r$, because $\alpha=W(x)\le\alpha$ while $r=W_1(x)\nless r$. Similarly, if we suppose that 
 $W (x)> \bar g(W_1(x))$, namely
 \bel{contr2}
\bar{\alpha}< \alpha, 
\eeq
by the definition of  $\bar{\alpha}$ we get that, for every $\alpha'\in(\bar \alpha,\alpha)$, one should have $\cS^{1}_r\subseteq \cS_{\alpha'}$.  But this contradicts the fact that $x\in \cS^{1}_r$, since $r=W_1(x)\le r$, while $x\notin  \cS_{\alpha'}$, being $\alpha=W(x)\nleq \alpha'$.
\end{proof}

\vsm 
We are now ready to  show that, given a $\p$-MRF $W$ with $\p>0$,   the control system \eqref{Egen} is sample stabilizable to $\C$ with $(p_0,W)$-regulated cost. For any pair $r$, $R>0$ with $r< R$, let us   set
\bel{mur}
\begin{array}{l}
\hat\mu(r):=\underline{g}_{W,\d}(r)=\sup\left\{\mu>0: \ \ \{\tilde z: \ W(\tilde z)\le\mu\}\subseteq\overset{\circ}{B}_r(\C)\right\}, \\ \, \\ 
 \sigma(R):=\bar {g}_{W,\d}(R)=\inf\left\{\sigma>0: \ \ \{\tilde z: \ W(\tilde z)\le\sigma\}\supseteq B_R(\C)\right\}. 
\end{array}
\eeq
By Lemma \ref{LW=d},  if $r<{\d(\tilde z)}\leq R$,   then $\tilde z\in W^{-1}((\hat \mu(r), \sigma(R)])$ and the values $\hat\mu(r)$,  $ \sigma(R)$ are finite  and  verify $0<\hat\mu(r)< \sigma(R)$.    Let us choose 
\bel{delta}
\delta=\delta(r,R):=\bar\delta(\hat\mu(r),\sigma(R)),
\eeq
  where $\bar\delta(\hat\mu,\sigma)$ is defined by \eqref{conddelta}. Fixed $\varepsilon>0$, for instance, $\varepsilon=1$,  by Propositions \ref{3.5}, \ref{pinvariant} it follows that for every partition $\pi=(t^j)$ of 
$[0,+\infty)$ with  $\text{diam}(\pi)\le\delta$ and for every initial state $z\in\R^n\setminus\C$ such that ${\bf d}(z)\le R$, any  $\pi$-sampling  trajectory-control pair  $(x,u)$  of   \eqref{Egen} with $x(0)=z$  has $x$ defined in $[0,+\infty)$  and verifies:
\begin{itemize}
\item[(i)] $\bar t:=T_x^{\hat\mu(r)/4}<+\infty$;
\item[(ii)] for every  $t\in[0,\bar t)$ and $j\ge1$   such that  $t\in[t^{j-1}, t^j)$,
\bel{estW}
  W(x(t))-W(x(t^{j-1}))+\p  \int_{t^{j-1}}^t l(x(\tau),u(\tau))\,d\tau  
\le -\frac{\gamma(W(x(t^{j-1}))) }{2}(t -t^{j-1});
\eeq
\item[(iii)] for every  $t\ge\bar t$, $W(x(t))\le\hat\mu(r)$, which implies that $\d(x(t))\le r$.
\end{itemize}
The time $\bar t$ might be zero when $\d(z)\le r$. Of course,  condition  (ii) is significant only if $\bar t>0$.

Observing that \eqref{estW} implies 
 \noindent \bel{p}
  W(x(t))-W(z)\le -\frac{\gamma(W(x(t^{j-1}))) }{2}t \qquad\forall t\in[0,\bar t),
\eeq
the construction of a 
${\mathcal {KL}}$ function $\beta$ such that 
\bel{betaproof}
\d(y(t))\leq \beta(\d(z),t) \quad \forall t\in[0,\bar t) 
\eeq
can  be obtained arguing as  in \cite[p.600]{LMR16}, hence we omit it.  Together with (iii), this yields that 
$$
\d(x(t))\leq \max\{\beta(\d(z),t),r\} \quad \forall t \ge0.
$$
Moreover, when $\bar t>0$ by  summing up $j$ from 0 to  the last index $\tilde n$ such that $t^{\tilde n}< \bar t$,  from (ii) it follows that 
\bel{estt}
W(x(\bar t))-W(z)+\p\int_0^{\bar t} l(x(\tau),u(\tau))\,d\tau  \le -\frac{\gamma(W(x(t^{\tilde n}))) }{2}\bar t.
\eeq
Hence 
$$
\int_0^{\bar t} l(x(\tau),u(\tau))\,d\tau  \le \frac{W(z)}{\p}
$$
and this concludes the proof  since 
$$
\bar t\ge  \bar T_x^r=\inf\{t>0: \ \d(x(\tau))\le r \ \ \forall \tau\ge t\}.
$$ 
\qed
\begin{Remark}\label{RT}{\rm When  $\d(z)>r$,  the time  $ \bar T_x^r$, after which  any $(r,R)$ stable $\pi$-sampling trajectory $x$ starting from $z$ remains definitively in $B_r(\C)$,   is uniformly bounded  by a positive constant. Precisely,  using the above notations, by the previous proof   one can easily deduce the following upper bound  
\bel{BtU}
\bar T_x^r\le \bar t\le \frac{2(W(z)-W(x(\bar t)))}{\gamma(W(x(\bar t)))}= \frac{2\left(W(z)-\frac{\hat\mu(r)}{4}\right)}{\gamma(\hat\mu(r)/4)}.
\eeq
}
 \end{Remark}
\vsm
\subsection{Proof of the Euler  stabilizability with $(p_0,W)$-regulated cost} \label{3.2}
 Let us start with some preliminary results. In the sequel we make use of all  the notations introduced in the previous subsection.
\vsm 
  The following lemma  establishes a {\it uniform} lower bound for the time needed to admissible trajectories  starting from the same point  $z$ and approaching the target, to reach an $\varepsilon$-neighborhood  of the target.  
 
 \begin{Lemma}\label{Leps} Assume {\rm (H0)}. Given  $R>0$,   let us set
$$
\ds\tilde M(R):=\sup\{|f(x,u)|: \ \ x\in B_R(\C)\setminus \C, \ \ u\in U\}.
$$
Then  for any  $z\in\R^n\setminus\C$   such that $\d(z)\le R$ and $\varepsilon\in(0,\d(z))$,  
setting
\bel{Teps1}
T_\varepsilon:= \frac{\d(z)-\varepsilon}{\tilde M(R)}>0,
\eeq
  every  admissible trajectory-control pair $(x,u)\in {\mathcal{ A}}_f({ z})$ with $T_x\le+\infty$ and such that $\lim_{t\to T^-_x}\d(x(t))=0$,  verifies 
\bel{Teps}
\d(x(t))\ge\varepsilon \quad \forall t\in[0,T_\varepsilon].
\eeq 
As a consequence, $T_x\ge \frac{\d(z)}{\tilde M(R)}$.
\end{Lemma}
\begin{proof} Given  $(x,u)\in {\mathcal{ A}}_f({ z})$ as above, let us  set $\bar \tau:=\sup\{t\ge 0: \  \d(x(t))\ge\d(z)\}$. The time  $\bar \tau$ is clearly finite and defining
  $\tilde T_x^\varepsilon:=\inf\{t>\bar \tau: \ \d(x(t))\le\varepsilon\}$, one trivially has $0\le \bar \tau< \tilde T_x^\varepsilon$ and  $x([\bar \tau, \tilde T_x^\varepsilon])\subseteq \overline{B_R(\C)}\setminus \C$.   If   $\bar z^\varepsilon\in\partial\C$   verifies
$$
\varepsilon=\d(x(\tilde T_x^\varepsilon))=|x(\tilde T_x^\varepsilon)-\bar z^\varepsilon|,
$$
then the uniform bound \eqref{Teps} is a consequence of the following inequalities
$$
\d(z)=\d(x(\bar\tau))\le|x(\bar\tau)-\bar z^\varepsilon|\le |x(\bar\tau)-x(\tilde T_x^\varepsilon)|+|x(\tilde T_x^\varepsilon)-\bar z^\varepsilon|\le \tilde M(R  )\,\tilde T_x^\varepsilon+\varepsilon, 
$$
  implying that   $\tilde T_x^\varepsilon\ge  \frac{\d(z)-\varepsilon}{\tilde M(R)}= T_\varepsilon$. Indeed, $\d(x(t))\ge\d(z)>\varepsilon$ for all $t\in[0,\bar \tau]$ and  the definition of $\tilde T_x^\varepsilon$ implies that $ \d(x(t))>\varepsilon$ for all $t\in]\bar \tau, \tilde T_x^\varepsilon]$, so that $\d(x(t))\ge\varepsilon$  for all $t\in[0,T_\varepsilon]$. By the arbitrariness of $\varepsilon>0$, this implies that $T_x\le \d(z)/\tilde M(R  )$.
\end{proof} 

Next result allows us to determine, given a $\p$-MRF $W$, a positive constant $R$ and a sampling time $\delta>0$ small enough,  a radius  $r<R$ such that any $\pi$-sampling trajectory-control pair for \eqref{sis1} with initial point $z$ verifying $\d(z)\le R$ and with diam$(\pi)=\delta$ is $(r,R)$-stable.

 \begin{Lemma}\label{Lrdelta} Assume {\rm (H0)}. Let $W$ be a $\p$-MRF with $\p\ge0$  and for any pair $r$, $R>0$ with $r< R$, let $\delta=\delta(r,R)$  be   defined accordingly to \eqref{delta}. Then,  for every fixed $R>0$,  $\delta(\cdot,R)$ 
is positive  and increasing  and 
 $$
 \lim_{r\to 0^+}\delta(r,R)=0, \qquad \delta(R):=\lim_{r\to R^-}\delta(r,R)<+\infty.  
 $$
  \end{Lemma} 
  
  \begin{proof}	
  By  Subsection \ref{Subsample},  we have that 
  	$$\delta(r,R)=\bar \delta(\mu(r),\sigma(R)),$$
  	where $\bar \delta$ is defined as in  \eqref{conddelta}, in  Proposition \ref{pinvariant}. Since the map $r\mapsto\hat \mu(r)$ is increasing, $\hat \mu(r)$ vanishes as $r\to 0^+$ and $\hat \mu(r)$ is bounded by $\sigma(R)$ as $r\to R^-$,    to conclude it suffices to show that for every $\sigma>0$ the map $\hat \mu\mapsto \bar\delta(\hat \mu, \sigma)$ (a) is increasing in $(0,\sigma)$, (b) vanishes in $0$ and (c) is bounded as $\hat\mu$ tends to $\hat\mu(R)$.	Let $L(\hat\mu,\sigma)$ be the Lipschitz constant of $W$ on $W^{-1}([\hat \mu,2\sigma])$, let $m=m(\sigma)$ be  as in \eqref{m1} and  recall from \eqref{conddelta} the following  definition 
  	$$\bar\delta(\hat\mu,\sigma)=\min\left\{\hat\delta\left(\frac{\hat\mu}{4},2\sigma\right), \frac{\hat \mu}{4L(\hat \mu,\sigma)\,m}\right\}.$$
  	  We note that $\hat\mu\mapsto L(\hat\mu,\sigma)$ is decreasing in $(0,\sigma)$: this implies at once conditions (b) and (c) and the fact that, for every $\sigma>0$, the map $\hat \mu\mapsto \hat \mu/4L(\hat \mu,\sigma)m$ is increasing.  
  	  To conclude it is left to show that, for every $\sigma>0$, the map $\hat \mu\mapsto \hat \delta(\hat \mu,\sigma)$ is increasing in $(0,\sigma)$. This monotonicity that can be easily derived  arguing as above, by the definitions of the constants in the proof of  \cite[Proposition 3.5]{LMR16}, hence we omit to prove it.
 \end{proof} 
  
Owing to Lemma \ref{Lrdelta}, given a $\p$-MRF $W$  and a positive constant  $R$,  we can assume without loss of generality that $\delta(\cdot,R)$ defined as above is strictly increasing and continuous. Therefore,  for any $R>0$ we can define the inverse of the map  $r\mapsto \delta(r):=\delta(r ,R)$, given by
\bel{rdelta}
 \delta\mapsto r(\delta) \qquad\forall \delta\in[0,\delta(R)],
\eeq
  which is continuous, strictly increasing and  such that $r(0)=0$ and $r(\delta(R))=R$.  
   As an immediate consequence,  by the sample stabilizability of   \eqref{Egen} with $(p_0,W)$-regulated cost we get the following result.
  \begin{Lemma}\label{Lrdelta2} Assume {\rm (H0)} and let $W$ be a $\p$-MRF with $\p\ge0$. Then there exists a function $\beta\in{\mathcal {KL}}$ such that,  for each pair $R>0$  and $\delta\in(0,\delta(R))$,  
for every partition $\pi$ of 
$[0,+\infty)$ with  $\text{\em diam}(\pi)=\delta$ and for any initial state 
$z\in\R^n\setminus\C$ such that ${\bf d}(z)\le R$, any  $\pi$-sampling  
trajectory-control pair  $(x,u)$  of   \eqref{Egen} from $z$ is  defined in $[0,+\infty)$  and verifies:
\bel{betaSdelta}
{\bf d}(x(t))\le\max\{\beta(R,t), r(\delta)\} \qquad \forall t\ge0.
\eeq
Moreover, if $\p>0$, 
\bel{estWcdelta}
x^0(\bar T_x^{r(\delta)})=\int_0^{\bar T_x^{r(\delta)}}l(x(\tau),u(\tau))\,d\tau\le  \frac{W(z)}{\p},
\eeq
where
 $
\bar T_x^{r(\delta)}=\inf\{t>0: \ \d(x(\tau))\le r(\delta) \ \ \forall \tau\ge t\}
 $, as in   \eqref{bar T}. 
\end{Lemma}
 
 \begin{Remark}\label{Rex}{\rm When $f$ and $l$ verify hypothesis (H0) and  system \eqref{Egen} is sample stabilizable to $\C$ with $(p_0,W)$-regulated cost,  there always exist continuous  Euler solutions to \eqref{Egen}. Indeed,   for any  $z$ with $0<\d(z)\le R$ and  any  sequence $(x_i,u_i)$ of $\pi_i$-sampling trajectory-control pairs of \eqref{Egen} with $x(0)=z$, $\delta_i:=$diam$(\pi_i)\to0$ as $i\to+\infty$,  and associated costs $x^0_i$, it turns out that $(x_i^0, x_i)$  is equi-Lipschitz  continuous on $[0,+\infty)$ with Lipschitz  constant $m>0$. Hence   the existence of continuous, actually $m$-Lipschitz continuous Euler solutions to \eqref{Egen} from $z$ with $m$-Lipschitz continuous Euler costs follows straightforwardly by Ascoli-Arzel\'a's Theorem.}
\end{Remark}

\vsm  
We are now in position to  prove that, if we assume (H0) and $W$ is a $\p$-MRF with $\p>0$,  the feedback $K$  Euler-stabilizes the system \eqref{Egen} to $\C$ with $(p_0,W)$-regulated cost.  Given $z\in\R^n\setminus\C$, let  $(\x^0,\x)$ be an Euler solution of \eqref{Egen} with initial condition $\x(0)=z$. 
 By definition there exist a sequence of partitions $(\pi_i)$ of $[0,+\infty)$
such that $\delta_i:=\text{diam}(\pi_i)\to 0$ as $i\to \infty$, a sequence of $\pi_i$-sampling  trajectory-control pairs $(x_i,u_i)$ for  \eqref{Egen}  with  $x_i(0)=z$ for each $i$, and associated costs $x^0_i$,  satisfying   
\bel{hpE}
(x_i^0,x_i)\to  (\x^0,\x) \ \text{locally uniformly on $[0,+\infty)$.}  
\eeq
 Set $R:=\d(z)$ and let $\beta$, $\delta(R)$ and $r:[0,\delta(R)]\to[0,R]$ be as in Lemma \ref{Lrdelta2}. Since $\delta_i\to 0$, we can assume without loss of generality that $\delta_i<\delta(R)$ for all $i$. Hence Lemma \ref{Lrdelta2}  implies that, for every $i$, 
\bel{euleri1}
\d(x_i(t))\leq \max\{\beta(\d(z),t),r(\delta_i)\} \qquad\forall t\ge0
\eeq
and
\bel{euleri2}
x_i^0(t)\le  \frac{W(z)}{\p} \quad \forall t\in[0,\bar T_{x_i}^{r(\delta_i)}].
\eeq
As  $i\to \infty$, we have that $\delta_i\to 0$ and  consequently  $r(\delta_i)\to 0$.  Then by \eqref{hpE} and  \eqref{euleri1}  we obtain that
\bel{thE0}
\d(\x(t))\leq \beta(\d(z),t) \qquad \forall t\ge 0.
\eeq
Hence $\displaystyle\lim_{t\to+\infty}\d(\x(t))=0$ and there exists 
$$
T_\x:=\inf\{\tau\ge 0: \ \ \lim_{t\to \tau^-}\d(\x(t))=0\}\le +\infty.
$$
 To conclude the proof it remains only  to show that 
\bel{thE}
 \lim_{t\to T_\x^-} \x^0(t) \le  \frac{W(z)}{\p},
\eeq 
where the limit  is well defined, since $\x^0$,  pointwise limit of  monotone nondecreasing functions,  is monotone nondecreasing.
Passing eventually to a subsequence, we set $\bar T:=\lim_i\bar T_{x_i}^{r(\delta_i)}$.  In view of Lemma \ref{Leps},   $\bar T$ satisfies
\bel{barTpos}
\bar T\ge \frac{\d(z)}{m}>0.
\eeq
  Then for any $t\in[0,\bar T)$ one has   $\bar T_{x_i}^{r(\delta_i)}>t$ for all $i$ sufficiently large and,   taking the limit as $i\to\infty$ in  \eqref{euleri2}, by \eqref{hpE} it follows that 
\bel{estbarT}
\x^0(t)\le  \frac{W(z)}{\p} \qquad\forall t\in[0,\bar T). 
\eeq
If $\bar T=+\infty$,  this implies directly the thesis  \eqref{thE}.   If instead  $\bar T<+\infty$,  the definition of $\bar T_{x_i}^{r(\delta_i)}$ yields that  
$$
\d(x_i(\bar T_{x_i}^{r(\delta_i)}))= r(\delta_i).
$$
Moreover,  by the  locally uniform convergence of $x_i$ to $\x$ and the  $m$-Lipschitz continuity of   $\x$,   we get the following estimate 
$$
\begin{array}{l}
\d(\x(\bar T))\le
 |\x(\bar T)-\x(\bar T_{x_i}^{r(\delta_i)})| +|\x(\bar T_{x_i}^{r(\delta_i)})-x_i(\bar T_{x_i}^{r(\delta_i)})|+\d(x_i(\bar T_{x_i}^{r(\delta_i)})), \end{array}
$$
where the r.h.s. tends to zero as $i\to+\infty$. Thus  we have in any case
\bel{Tpos}
0< \frac{\d(z)}{m}\le  T_\x\le \bar T, 
\eeq
where the first inequality is again a consequence of the $m$-Lipschitz continuity of $\x$.  This concludes the proof, since   \eqref{estbarT} implies now the thesis \eqref{thE}.  \qed
 
%%%%%%%%%%%%%%%%%%%%%%%%%%%%%%%%%%%%%%%%%%%%%%%%%%%%%%%%%%%%%%%%%%%%%%%%%%%
\section{On the notion of $\p$-Minimum Restraint Function}\label{RMRF}
In Subsection \ref{MRF1} we  obtain an equivalent formulation of the definition of $\p$-MRF. Using this condition,  in Subsection \ref{MRFL} we prove that, when the data $f$ and $l$ are locally Lipschitz continuous, the existence of a  locally Lipschitz, not necessarily semiconcave, $\p$-MRF $W$, still guarantees sample and Euler stabilizability of the control system \eqref{Egen} to $\C$ with $(p_0,W)$-regulated cost. Subsection \ref{SubRMRF} is devoted to extend these results to the original notion of $\p$-MRF  introduced in \cite{MR13}.
 
\subsection{An equivalent notion of $\p$-MRF}\label{MRF1} 
 
\begin{Proposition}\label{Wweak}
Assume {\rm (H0)}.  Let $W :\ol{\R^n\setminus \T}\to[0,+\infty)$, be a continuous function, which is  positive definite,  proper, and semiconcave on $\R^n\setminus\T$. 
Then $W$ is a $\p$-MRF for some $\p\ge0$ if and only if {for any}   continuous function $W_1 :\ol{\R^n\setminus \T}\to[0,+\infty)$, positive definite, and  proper on $\R^n\setminus\T$, {there is some}  continuous, strictly increasing function $\tilde\gamma:(0,+\infty)\to(0,+\infty)$ such that 
 $$ 
 H (x,\p, D^*W(x) )\le-\tilde\gamma(W_1(x)) \quad \forall x\in {{\R^n}\setminus\T}.      
$$
\end{Proposition}
 For instance,  one can   choose $W_1=\d$,   distance function from  $\C$.
\vsm
Proposition \ref{Wweak} is a consequence of the following more general result, involving  a locally Lipschitz,   not necessarily semiconcave, function  $W$. 
\begin{Proposition}\label{W=d} 
Assume {\rm (H0)} and let    $W:\ol{\R^n\setminus {\T}}\to[0,+\infty)$ be  a continuous map, which is  locally Lipschitz, positive definite, and  proper on $\R^n\setminus\T$.  Given $\p\ge 0$ and a set  $\Omega\subseteq {{\R^n}\setminus\T}$,   the following statements are equivalent:
\begin{itemize}
\item[(i)] $W$ verifies the decrease condition 
  \bel{LMRH} 
    H (x,\p,  \partial_PW(x))\le-\gamma(W(x)) \quad \forall x\in \Omega,     
\eeq
for some continuous, strictly increasing function $\gamma:(0,+\infty)\to(0,+\infty)$;
\item[(ii)]
for any continuous function $W_1 :\ol{\R^n\setminus \T}\to[0,+\infty)$ which is  positive definite  and  proper on $\R^n\setminus\T$,  there exists  some continuous, strictly increasing function $\tilde\gamma:(0,+\infty)\to(0,+\infty)$ such that $W$ verifies
\bel{PMRH} 
    H (x,\p,  \partial_PW(x) )\le-\tilde\gamma(W_1(x)) \quad \forall x\in \Omega.      
\eeq 
\end{itemize}
\end{Proposition}
\begin{proof} Let us first prove that (i) $\Rightarrow$ (ii). Assume  that  $W$ verifies the decrease condition \eqref{LMRH}.  
Given an arbitrary continuous  map $W_1$, positive definite and proper in $\R^n\setminus\C$, let   $\tilde\gamma:(0,+\infty)\to(0,+\infty)$ be a continuous, strictly increasing approximation from below of the increasing map $r\mapsto \gamma\circ \underline{g}_{\, W,W_1}(r)$, where $\underline{g}_{\, W,W_1}$ is defined accordingly to Lemma \ref{LW=d}. Then by \eqref{Ldis}, for any $x\in \R^n\setminus\C$, one has
$$
W(x)\ge \underline{g}_{\, W,W_1}(W_1(x)) \ \Longrightarrow \  \gamma(W(x))\ge \gamma\circ\underline{g}_{\, W,W_1}(W_1(x))\ge \tilde\gamma(W_1(x)),
$$
so that \eqref{LMRH}  implies \eqref{PMRH} for such  $\tilde\gamma$. 

To prove that (ii) $\Rightarrow$ (i), it is enough to invert the roles of $W$ and $W_1$. Precisely, if  \eqref{PMRH} is verified for some $W$,  $W_1$ and  $\tilde\gamma$ as in the statement of the proposition,  arguing as above one obtains that  $W$  verifies  \eqref{LMRH} choosing as $\gamma:(0,+\infty)\to(0,+\infty)$ any continuous, strictly increasing approximation from below of the increasing map $r\mapsto \tilde\gamma\circ \underline{g}_{\,  W_1,W}(r)$.
\end{proof}

\vsm  
\begin{proof}[Proof of Proposition \ref{Wweak}]   The only non trivial fact in order to derive  Proposition \ref{Wweak} from  Proposition \ref{W=d}, is that \eqref{LMRH} involves the proximal subdifferential $\partial_PW(x)$ at $x$  instead of the set of limiting gradients  $D^*W(x)$ at $x$, considered in the decrease condition for a $\p$-MRF. However, when $W$ is locally Lipschitz continuous, condition  \eqref{LMRH} with $\Omega=\R^n\setminus\C$  implies  readily the following inequality:
$$
H (x,\p,  \partial_LW(x))\le-\gamma(W(x)) \quad \forall x\in {{\R^n}\setminus\T},     
$$
where $\partial_LW(x)$ denotes the limiting subdifferential at $x$.  This concludes the proof, since a $\p$-MRF $W$ is locally semiconcave and therefore $\partial_LW(x)=D^*W(x)$ at any $x$. \end{proof}

 \subsection{Lipschitz continuous $\p$-MRF}\label{MRFL} 
Under  the following hypothesis:
\vsm
\noindent {\bf (H1)} The sets $U\subset\R^m$, $\T\subset \R^n$ are  closed and the boundary  $\partial\T$ is compact.   $f:\overline{(\R^n\setminus\C)}\times U\to\R^n$,   $l: \overline{(\R^n\setminus\C)}\times U\to[0,+\infty)$ are continuous functions such that for  every compact subset ${\mathcal K}\subset \overline{\R^n\setminus\C}$ there exist  $M_f$, $M_l$,  $L_f$, $L_l>0$ such that
$$
\left\{\begin{array}{l}
|f(x,u)|\le M_f, \quad l(x,u)\le M_l \qquad \forall (x,u)\in {\mathcal K}\times U,  \\ [1.5ex]
|f(x_1,u)-f(x_2,u)|\le L_f|x_1-x_2|,  \\ [1.5ex]
 |l(x_1,u)-l(x_2,u)|\le L_l |x_1-x_2|\qquad \forall (x_1,u), \, (x_2,u)\in {\mathcal K}\times U,
\end{array}\right.
$$
\vsm 
 \noindent   we obtain the main result of this section:

  \begin{Theorem}\label{TLip} Assume {\rm (H1)} and let $\p\ge 0$. Let $W:\overline{ \R^n\setminus\C }\to[0,+\infty)$ be a locally Lipschitz continuous map on $\overline{\R^n\setminus\C}$,  such that $W$ is positive definite, and  proper on $\R^n\setminus\T$, and verifies the decrease condition\begin{equation}\label{TLipCondition}
    H (x,\p,  \partial_PW(x) )\le-\tilde\gamma(W_1(x)) \quad \forall x\in {{\R^n}\setminus\T},      
\end{equation}
  for some continuous, strictly increasing function $\tilde\gamma:(0,+\infty)\to(0,+\infty)$  and some continuous function $W_1 :\ol{\R^n\setminus \T}\to[0,+\infty)$, positive definite, and  proper on $\R^n\setminus\T$. Then there exists a $\ds\frac{\p}{2}$-MRF $\bar W$, which also satisfies $\bar W(x)\leq W(x)$ for all $x\in \R^n\setminus\T$. \end{Theorem}

Theorem  \ref{TLip}, whose proof is postponed to  Appendix \ref{A1},   generalizes the result on the existence of a semiconcave Control Lyapunov Function  obtained in  \cite[sect. 5]{R00}  to the present case, where the decrease condition involves also the Lagrangian $l$ and the target is not the origin, but an arbitrary closed set $\C$ with compact boundary.
\vsm
Let us  call a map $W$ as in Theorem \ref{TLip} a {\it Lipschitz continuous $\p$-MRF}.
 As an immediate consequence of Theorems \ref{TLip} and \ref{Tintro1},   we have the following: 
   \begin{Corollary}\label{CLip} Assume {\rm (H1)} and let $\p> 0$. Let $W:\overline{ R^n\setminus\C }\to[0,+\infty)$ be a Lipschitz continuous  $\p$-MRF.   Then there exists a locally bounded feedback $K:\R^n\setminus\C\to U$  that sample and Euler stabilizes  system \eqref{Eintro}   
 to $\C$ with {$(p_0/2,W)$}-regulated cost. 
\end{Corollary}
 Note that, {using the notation of Theorem  \ref{TLip},}   the feedback $K$  {in Corollary \ref{CLip}} is actually a $\bar W$-feedback and the claim above relies on  the inequality $\bar W\le W$.
 
\subsection{Comparison with the original notion of $\p$-MRF}\label{SubRMRF} 
Let us call $\p$-OMRF   the notion of $\p$-MRF originally introduced in \cite{MR13}.  
 \begin{Definition}[$\p$-OMRF]\label{defMRF}
   Let $W:\ol{\R^n\setminus {\T}}\to[0,+\infty)$ be a continuous function, and let us assume that $W$ is  locally semiconcave, positive definite, and  proper on $\R^n\setminus\T$. We say that $W$  is a \emph{$\p$-OMRF  for some $\p\ge0$}  if it verifies
\bel{MRHv} 
H (x,\p, D^*W(x) )<0 \quad \forall x\in {{\R^n}\setminus\T}.
\eeq
\end{Definition}    
A $\p$-MRF   is obviously  a $\p$-OMRF,  but the converse might be false.  By  \cite{MR13} we have the following result.
 
 \begin{Proposition}\label{3.2} {\rm \cite[Prop. 3.1]{MR13}} Assume that $W$ is a $\p$-OMRF with  $\p\ge0$. Then for  every $\sigma>0$  there exists a  continuous, increasing   map $\gamma_\sigma:(0,\sigma]\to (0,+\infty)$ such that   
  \bel{c0}
  H  (x,\p,D^*W(x)) <-\gamma_\sigma(W(x)) \qquad \forall x\in W^{-1}((0,\sigma]).
  \eeq
\end{Proposition}

Proposition \ref{3.2} clarifies  the   difference between the two notions:  the existence of a  $\p$-OMRF implies that there exists a  rate function $\gamma_\sigma$, which is  in general,  {\it not global}. In particular, $\gamma_\sigma$ can become smaller and smaller as $\sigma$ tends  to $+\infty$. Consequently, also  the feedback $K$ can be defined only  given  a $\sigma>0$, on $W^{-1}((0,\sigma])$. 

\begin{Remark} {\rm 
%		{Ho spostato  questo Remark: ero dopo del Teorema}
	When a $\p$-OMRF $W$ verifies condition \eqref{MRHv}  in the following stronger form
\bel{MRHvs} 
\forall M>0: \quad \sup_{ p\in D^*W(x)} H(x,\p, p)<0 \quad \forall x\in {{\R^n}\setminus\T} \text{ s.t. } \d(x)\ge M \,, 
\eeq
 it is not difficult to prove  that,  under the assumptions of Proposition \ref{3.2},   there exists a continuous, strictly increasing function $\gamma:(0,+\infty)\to(0,+\infty)$  independent of $\sigma$,  such that \eqref{c0}   holds for all $x\in\R^n\setminus\C$  (see \cite[Remark 3.1]{MR13}). 
In other words, 
$$
 \text{{\it a  $\p$-OMRF $W$ verifying \eqref{MRHvs}  is actually  a $\p$-MRF.}}
 $$
 }
\end{Remark}
 
 The proof of  Theorem  \ref{Tintro1} can be easily adapted to derive the following  result.

\begin{Theorem}\label{T1O}
Assume that $f$, $l$ verify hypothesis {\rm (H0)} and let $W$ be a $\p$-OMRF   with $\p>0$. Then for any $\sigma>0$ there exists a locally bounded feedback $K:W^{-1}((0,\sigma])\to U$  that sample and Euler stabilizes  system \eqref{Egen}   
 to $\C$ with $(p_0,W)$-regulated cost for any initial point   $z\in W^{-1}((0,\sigma])$. 
\end{Theorem}

As in the case of $\p$-MRF, when  $f$, $l$ are  locally Lipschitz continuous in $x$, we can replace the semiconcavity assumption in the definition of a $\p$-OMRF with local  Lipschitz continuity. Precisely, we establish what follows.
    \begin{Theorem}\label{CLipO} Assume {\rm (H1)} and let $\p\ge 0$. Let $W:\overline{\R^n\setminus\C}\to[0,+\infty)$ be a locally Lipschitz continuous map on $\overline{(\R^n\setminus\C)}$, such that $W$ is positive definite, and  proper on $\R^n\setminus\T$, and verifies the decrease condition 
   	 \bel{LOMRH} 
   	H (x,\p,  \partial_LW(x))<0 \quad \forall x\in {{\R^n}\setminus\T}     
   	\eeq
   	  Then there exists a $\ds\frac{\p}{2}$-OMRF $\bar W$ which also satisfies $\bar W(x)\leq W(x)$ for all $x\in \R^n\setminus\T$. 
    	 \end{Theorem}
	   The proof of this theorem   is sketched in Appendix \ref{pCLipO}.

	  \vsm
Let us  call a map $W$ as in Theorem {\ref{CLipO}} a {\it Lipschitz continuous $\p$-OMRF}.
 Theorems \ref{CLipO} and  \ref{T1O} imply what follows:
   \begin{Corollary}\label{OLip} Assume {\rm (H1)} and let $\p> 0$. Let $W:\overline{ \R^n\setminus\C }\to[0,+\infty)$ be a Lipschitz continuous  $\p$-OMRF.  Then {for any $\sigma>0$} there exists a locally bounded feedback $K:  W^{-1}((0,\sigma])\to U$  that sample and Euler stabilizes  system \eqref{Egen}   
    	  to $\C$ with {$(p_0/2,W)$}-regulated cost for any initial point   $z\in W^{-1}((0,\sigma])$.  
\end{Corollary}

\section{An example: stabilization of  the non-holonomic integrator control  system with regulated cost}\label{SecEx}
Let us illustrate the preceding theory through a classical example. Precisely, in the first part of this section we provide a $\p$-MRF $W_1$ for the non-holonomic integrator control system associated to a Lagrangian $l$, that verifies a suitable growth condition (see \eqref{lV} below). In view of  Theorem \ref{Tintro1},  this implies the existence of a possibly discontinuous feedback $K$ that sample and Euler stabilizes the non-holonomic integrator to the origin with the cost bounded above by $W_1/\p$. Furthermore, we show  how weakening the  requirements on the $\p$-MRF (by replacing semiconcavity with Lipschitz continuity) may be crucial for the effective construction of a $\p$-MRF.  In particular,  for any bounded Lagrangian $l$ which  cannot satisfy assumption  \eqref{lV} below  for any $C>0$ (as,  for instance,  in the case of the minimum time problem,  where $l\equiv1$),  we   provide a less regular, Lipschitz continuous but not semiconcave   $\p$-MRF $W_2$.  In this case, the sample and Euler stabilizability of the control system with $({\p/2},W_2)$-regulated cost is guaranteed by Corollary \ref{CLip}.  

\vsm
Set $U:=\{u=(u_1,u_2)\in\R^2: \ \ u_1^2+u_2^2\le 1\}$,   $\T:=\{0\}\subset\R^3$ and  consider  the non-holonomic integrator control system:
\begin{equation}\label{enh}
\begin{cases}
\dot x_1= u_1\\
\dot x_2= u_2\\
\ds \dot x_3= x_1u_2-x_2u_1, \qquad u(t)=(u_1,u_2)(t)\in U.  
%\\\ds x(0)=(x_1,x_2,x_3)(0)=z\in \R^3\setminus\{0\}.
\end{cases}
\end{equation}
Given a  continuous Lagrangian $l(x,u)\ge0$, let us associate to  \eqref{enh} the cost
\begin{equation}\label{costoes}
\int_0^{T_x}l(x(t),u(t))\,dt 
\end{equation}
(as in the rest of the paper, $T_x$ denotes the exit-time of $x$ from  $\R^3\setminus\{0\}$).  Set $$
 f(x,u):=(u_1,u_2,x_1u_2-x_2u_1)  \qquad\forall (x,u)\in \R^3\times U.
 $$ 
The following map $W_1$,  introduced in   \cite{MRS04}, given by
  $$W_1(x):=\left(\sqrt{x_1^2+x_2^2}-|x_3|\right)^2+x_3^2 \qquad\forall x\in\R^3,$$
   is proper, positive definite,  locally semiconcave in $\R^3\setminus\{0\}$,  and verifies
$$\min_{u\in U} \langle p, f(x,u)\rangle=-\sqrt{V(x)}\qquad \forall x\in\R^3\setminus\{0\},   \  \forall  p\in D^*  W_1 (x),$$
where 
$$V(x):=\left(\sqrt{x_1^2+x_2^2}-|x_3|\right)^2+\left(\sqrt{x_1^2+x_2^2}-2|x_3|\right)^2(x_1^2+x_2^2) \quad\forall x\in\R^3.$$
Therefore,  $W_1$ is   a Control Lyapunov Function   for  the control system \eqref{enh} and, consequently, any $W_1$-feedback sample and Euler stabilizes \eqref{enh} to the origin \cite{R02}.
When the Lagrangian $l$ satisfies, for some positive constant $C$, 
\begin{equation}\label{lV}
0\leq l(x,u)\leq C \sqrt{V(x)}\quad\forall(x,u)\in (\R^3\setminus\{0\})\times U,
\end{equation}
then $W_1$ is also a $\p$-MRF  for every $\p\in(0,1/C)$. Indeed, for all $x\in \R^3\setminus \{0\}$ and  for all   $p\in D^*W_1(x)$, one has 
$$
\begin{array}{l}
\ds H(x,p_0,p)=\min_{u\in U}\{ \langle p, f(x,u)\rangle + \p\,l(x,u)\}\leq \\
\ds \qquad\qquad\qquad \min_{u\in U}\{ \langle p, f(x,u)\rangle \}+ p_0 C \sqrt{V(x)}=-(1-p_0 C)\sqrt{V(x)}.
 \end{array}
$$
However,   the Control Lyapunov Function $W_1$ cannot be a $\p$-MRF when  
\begin{equation}\label{lVlim}
\lim_{x\to0} \frac{\inf_{u\in U} l(x,u)}{\sqrt{V(x)}}=+\infty.
\end{equation}
 Since $V(x)$ tends to $0^+$ as $x\to0$,   condition \eqref{lVlim} does not hold,  for instance, for the minimum time problem, where $l\equiv 1$.

  For any bounded Lagrangian $l$, let $M_l>0$ verify $l(x,u)\le M_l$ for all $(x,u)\in (\R^n\setminus\C)\times U$. Then a discontinuous feedback that  sample and Euler stabilizes \eqref{enh} and at the meantime provides strategies for which the target is reached  with  regulated cost, can be obtained if we consider the  following Control Lyapunov Function $W_2$, introduced in \cite{Rthesis}:
$$ W_2(x):=\max\left\{\sqrt{x_1^2+x_2^2}, |x_3|-\sqrt{x_1^2+x_2^2}\right\} \qquad\forall x\in\R^3.$$
The map  $W_2$ is locally semiconcave only outside the set $S:=\{(x_1,x_2,x_3)\in \R^3\mid x_3^2=4(x_1^2+x_2^2)\}$, therefore,  it is not a $\p$-MRF. However, $W_2$ matches the weaker definition of \emph{Lipschitz continuous} $\p$-MRF for $\p<1/M_l$: it is indeed a locally Lipschitz continuous map in $\R^3$, which is positive definite and proper   in $\R^3\setminus \{0\}$,  and a direct computation  shows that 
$$H(x,p_0,p)\leq \min_{u\in U} \langle f
(x,u), p \rangle+ \p M_l<0 \ \ \forall x\in\R^3\setminus \{0\}, \ \forall p\in \partial_P W(x)$$
(see also \cite{R02} and \cite{MRS04}). 
Since   the data $f$ and $l$   verify assumption (H1),  it follows by Corollary \ref{CLip} that \eqref{enh}-\eqref{costoes} is sample and Euler stabilizable  with $(\l/2,W_2)$-regulated cost as soon as the Lagrangian $l$ is bounded.

%%%%%%%%%%%%%%%%%%%%%%%%%%%%%%%%%%%%%%%%%%%%%%%%%%%%%%%%%%%%%%%%%%%%%%%%%%%
\section{Conclusions}
In this paper we addressed sample and Euler stabilizability of nonlinear control system in an optimal control theoretic framework. We introduced the notion of sample and Euler trajectories with regulated cost, which conjugate stabilizability with an upper bound on the payoff,  depending on the initial state. Under mild regularity hypotheses on the vector field $f$ and on the Lagrangian $l$ and for a closed, possibly unbounded control set, we proved that the existence of a special Control Lyapunov Function $W$, called a $\p$-Minimum Restraint Function, $\p$-MRF,  implies that all sample and Euler  stabilizing trajectories have $(p_0,W)$-regulated costs. The proof is constructive:  it is based  indeed on the synthesis of appropriate feedbacks derived from $W$. As in the case of classical Control Lyapunov Functions, this construction requires that $W$ is locally semiconcave. However, by  generalizing an earlier result by Rifford \cite{R00} we established that it is possible to trade regularity assumptions on  $f$ and  $l$ with milder regularity assumptions on  $W$. In particular, we showed that if the vector field $f$ and the Lagrangian $l$ are locally Lipschitz   up to the boundary of the target,   then the existence of a mere locally Lipschitz $\p$-MRF $W$  provides  sample and Euler stabilizability with $(p_0/2,W)$-regulated cost.  

The present work is part of an ongoing, wider investigation of global asymptotic controllability and stabilizability in an optimal control perspective. A slightly weaker notion of $p_0$-MRF --called here $p_0$-OMRF-- was introduced in \cite{MR13} and further extended in \cite{LMR16} to more general optimization problems, in order to yield  global asymptotic controllability  with regulated cost. This paper represents the stability-oriented counterpart of  \cite{MR13}. In a forthcoming paper we will address the question of stabilizability with regulated cost for possibly non-coercive   optimization problems with unbounded  controls. Other interesting research directions include  the relation between $p_0$-MRFs and input-to-state stability (in the fashion of \cite{MRS04}) and the study of a possible inverse Lyapunov theorem for $p_0$-MRFs -- i.e., whether  the results in \cite{Sonn83}  can be extended to $p_0$-MRFs by showing that their existence is also a necessary condition for the global asymptotic controllability  of the control system with  regulated cost.      
 
\appendix
 \section{ }
\subsection{Proof of Theorem \ref{TLip}}\label{A1}
Let $W:\overline{\R^n\setminus\C}\to[0,+\infty)$ be a locally Lipschitz continuous map, positive definite and proper on  $\R^n\setminus\C$, and verifying the decrease condition \eqref{TLipCondition}, namely, such that 
\bel{LMRHA} 
H (x,\p,  \partial_PW(x))\le-\tilde\gamma(W_1(x)) \quad \forall x\in {{\R^n}\setminus\T}      
\eeq
for some strictly increasing, continuous  map $\tilde\gamma:(0,+\infty)\to(0,+\infty)$ and some continuous function $W_1:\overline{\R^n\setminus\C}\to[0,+\infty)$,  $W_1$  positive definite and proper on  $\R^n\setminus\C$.  Our goal is to show that there exists a  $\frac{\p}{2}$-MRF \, $\bar W$,  such that $\bar W\le W$. 
\vsm
The proof  is a careful adaptation of the arguments  in \cite[sect. 5]{R00}.  For this reason we explicitly prove  just the steps involving the decrease condition, where some changes are needed because of the presence of the Lagrangian $l$. 
\vsm 
Preliminarily, let us recall the notion of \emph{inf-convolution}  for  a  locally Lipschitz continuous   nonnegative  map  $g:\R^N\to\R$  and collect some useful properties (see e.g. \cite[Theorem 3.5.3, Lemma 3.5.7]{CS}, \cite[Section 1.5, Thm. 5.1]{CLSW}, \cite[Section II.4, Lemmas 4.11, 4.12]{BCD}).  
\begin{Lemma}\label{InfC} For any $\alpha>0$,  define
	$$
	g_\alpha(x):=\inf_{y\in\R^N}\left\{g(y)+\alpha|y-x|^2\right\} \qquad \forall x\in\R^N.
	$$
	Then $g_\alpha$ is locally semiconcave in $\R^N$ and 
	\begin{itemize}
		\item[{\rm (i)}]  for all $x\in\R^N$, there exists $\bar y\in\R^N$ such that  $g_\alpha(x)=g(\bar y)+\alpha|\bar y-x|^2$ (the above infimum is actually a minimum);
%		\mm{ In particular,  when $g$ is $L$-Lipschitz continuous in $\R^N$, then $|\bar y-x|\le   L/\alpha $ and $g_\alpha$ is $L$-Lipschitz continuous in $\R^N$}; 
		\item[{\rm (ii)}]  for all $x\in\R^N$, $0\le g_\alpha(x)\le g(x)$,  moreover $g_\alpha\nearrow g$ locally uniformly as $\alpha\to0$;  
		\item[{\rm (iii)}] for all $x\in\R^N$ such that  $\partial_P g_\alpha(x)$ is nonempty, $\bar y$ is unique and the proximal subgradient $\partial_P g_\alpha(x)$  is equal to  the singleton $\{2\alpha(x-\bar y)\}$, moreover, \newline $2\alpha(x-\bar y)\in\partial_P g(\bar y)$;
		\item[{\rm (iv)}] if $\Psi:\R\to\R$ is an increasing, locally semiconcave function, then $\Psi\circ g_\alpha$ is locally semiconcave;
		\item[{\rm (v)}] if $g,h:\R^N\to\R$ are semiconcave on $\Omega\subset\R^N$, then the function $\min\{g,h\}$ is semiconcave on $\Omega$.
	\end{itemize} 
\end{Lemma}
\vsm
{\sc Step 1.}  As it is not restrictive  in view of Proposition \ref{W=d}, let us assume  that  $W_1\equiv \d$.   We extend $\tilde\gamma$ continuously to $\R$, by setting $\tilde\gamma(t)=\tilde\gamma(0):=\lim_{s\to 0^+}\tilde\gamma(s)$ for every $t<0$.  Without loss of generality, we can suppose $\tilde\gamma$  $1$-Lipschitz continuous in $\R$. Otherwise, we can replace $\tilde \gamma$ in  the decrease condition \eqref{LMRHA}  with $\bar\gamma(t):=\inf_{s\in\R} \{\tilde \gamma(s)+|t-s|\}$ for every $t\in\R$. Indeed, it is not difficult  to see that $\bar\gamma\le\tilde\gamma$, and   $\bar\gamma$ is strictly increasing and  $1$-Lipschitz continuous. 
Therefore, the map
$$
\mathcal{W}:=\bar\gamma\circ W_1=\bar\gamma\circ \d
$$
%is locally semiconcave; moreover, it 
 is $1$-Lipschitz continuous and positive definite on $\R^n\setminus\C$. As a consequence of these results, \eqref{LMRHA}  implies that  $W$ verifies  
\bel{TMRH} 
H (x,\p, \partial_PW(x) )\le-\mathcal{W}(x) \qquad \forall x\in {{\R^n}\setminus\T}.      
\eeq 
\vsm
{\sc Step 2.}  For any integer $n\ge1$, let us  set 
$$
\begin{array}{c}
M_n:=\max\{W(x): \  x\in  B_1(W^{-1}([0,11n]))\}, \\ [1.5ex]
m_n:=\min\left\{\mathcal W(x): \  x\in  W^{-1}\left(\left[\frac{1}{2n},11n\right]\right)\right\}. 
\end{array}
$$
By the Lipschitz properties of $f$, $l$ and $W$, let us denote $L_f^n$, $L_l^n$, $L_W^n\ge1$ the Lipschitz constants of $f(\cdot,u)$, $l(\cdot,u)$ and $W$, respectively,  on the sublevel set $W^{-1}([0,M_n])$. Finally, let us set
\bel{aln}
\alpha_n:=\max\left\{8n(L^n_W)^2+1\,,\, \frac{2L^n_W(1+L^n_WL^n_f+\p L^n_l)}{m_n}+1\,,\, 11n\right\}.
\eeq
Let us extend $W$ to $\R^n$ by setting $W(x)=0$ for all $x$ in the interior of $\C$. For every $\alpha_n$, we define by inf-convolution the locally semiconcave  function $W_{\alpha_n}:\R^n\to[0,+\infty)$ as follows:
\bel{Wal}
W_{\alpha_n}(x):=\inf_{y\in\R^n}\left\{W(y)+\alpha_n|y-x|^2\right\} \qquad \forall x\in\R^n.
\eeq
\begin{Lemma}\label{L55}{\rm (\cite[Lemma 5.5]{R00})} Let $z\in W^{-1}([0,M_n])$. If the infimum in the definition of $W_{\alpha_n}(z)$ is attained at $\bar y$, then one has that $\bar y\in W^{-1}([0,M_n])$ and  \linebreak $|\bar y-z|\le\min\left\{\frac{1}{8nL_W^n}\,,\, \frac{m_n}{2(1+L^n_WL^n_f+\p L^n_l)} \right\} $;  moreover
	$$
	W(z)-\frac{1}{8n}\le W_{\alpha_n}(z)\le W(z).
	$$
\end{Lemma} 
\begin{Lemma}\label{L56} Let $z\in  W^{-1}\left(\left[\frac{1}{2n},11n\right]\right)$ and $p\in\partial_PW_{\alpha_n}(z)$. Then
	\bel{LMR56} 
	H (z,\p, p )\le-\frac{\mathcal{W}(z)}{2}     
	\eeq 
\end{Lemma} 
\begin{proof} Arguing similarly to the proof of \cite[Lemma 5.6]{R00},  by Lemmas \ref{L55} and \ref{InfC}, the infimum in the definition of $W_{\alpha_n}(z)$ is attained at a point $\bar y\in W^{-1}\left(\left[0,11n\right]\right)$,  
	verifying $|\bar y-z|\le \frac{m_n}{2(1+L^n_WL^n_f+\p L^n_l)}$ and such that  $p\in \partial_PW (\bar y)$. Therefore,   by the Lipschitz properties of $f$, $l$, $W$ and the $1$-Lipschitz continuity of $\mathcal{W}$ established in Step 1, we get
	$$
	\begin{array}{l}
	H (z,\p, p )=\inf_{u\in U}\left\{\langle p,f(z,u)\rangle+\p l(z,u)\right\}\le  \inf_{u\in U}\left\{\langle p,f(\bar y,u)\rangle+\p l(\bar y,u)\right\}    \\ [1.5ex]
	\qquad\qquad\qquad +\sup_{u\in U}\left(|p||f(z,u)-f(\bar y,u)|+\p |l(z,u)-l(\bar y,u)|\right)  \\ [1.5ex]
	\qquad\qquad\qquad\qquad \le -\mathcal{W}(\bar y)+L_W^nL_f^n|z-\bar y|+\p L_l^n|z-\bar y| \quad (\text{using  \eqref{TMRH}})  \\ [1.5ex]
	\qquad\qquad\qquad\qquad \le -\mathcal{W}(z)+(1+L_W^nL_f^n+\p L_l^n)|z-\bar y|   \\ [1.5ex]
	\qquad\qquad\qquad\qquad\ds \le -\mathcal{W}(z)+\frac{m_n}{2}\le -\frac{\mathcal{W}(z)}{2}.
	\end{array} 
	$$
\end{proof}

\vsm
{\sc Step 3.} Starting from  $(W_{\alpha_n})_{n\ge1}$, let us  construct a locally semiconcave $\ds\frac{\p}{2}$-MRF. 
\begin{Lemma}\label{L57}{\rm (\cite[Lemma 5.7]{R00})}  For each $n\ge1$, there exists an increasing, $C^\infty$, increasing map $\Psi_n:[0,+\infty)\to :0,+\infty)$ verifying the following properties.
	\begin{itemize}
		\item[{\rm (i)}] $\Psi_n(t)=t+\frac{1}{8n}$ for any $t\in \left[0,\frac{1}{2n}\right]$,
		\item[{\rm (ii)}] $\Psi_n(t)=t $ for any $t\in \left[\frac{1}{n}-\frac{1}{8n},10n\right]$,
		\item[{\rm (iii)}] $\Psi_n(t)\ge 11n+\max\{W(x): \ W_{\alpha_n}(x)\le t\}$ for any $t\in \Big[11n-\frac{1}{8n},+\infty\Big)$,
		\item[{\rm (iv)}] $\dot\Psi_n(t)\ge \frac{1}{2}$ for any $t\ge0$.
	\end{itemize}
\end{Lemma}
The function  $\bar W_n:=\Psi_n\circ W_{\alpha_n}$ is locally semiconcave on $\R^n$ by  Lemma \ref{InfC}, (iv). The required locally semiconcave $\ds\frac{\p}{2}$-MRF $\bar W$, is given by
$$
\bar W(x):=\min_{n\ge 1}\bar W_n(x) \qquad \forall x\in\R^n\setminus\C.
$$
Precisely, one has what follows.
\begin{Lemma}\label{L58}  For all integer $n\ge 1$ and for all  $z\in  W^{-1}\left(\left[\frac{1}{n},10n\right]\right)$, one has  $\bar W(z)=\min_{1\le k\le n}\bar W_k(z)$.
	Furthermore, if $p\in  \partial_P\bar W (z)$, then  
	\bel{LMR58} 
	H \left(z,\frac{\p}{2},p \right)\le-\frac{\mathcal{W}(z)}{4}.      
	\eeq 
\end{Lemma} 
\begin{proof}  For all $z\in  W^{-1}\left(\left[\frac{1}{n},10n\right]\right)$, the proof of the following facts:
	\begin{itemize}
		\item[(i)]     $\bar W(z)=\bar W_{n_0}(z):=\min_{1\le k\le n}\bar W_k(z)$ and
		\bel{subd}
		p\in\partial_P\bar W(z) \ \Longrightarrow \ p\in \partial_P\bar W_{n_0}(z)=\Psi'_{n_0}(W_{\alpha_{n_0}}(z))\partial_PW_{\alpha_{n_0}}(z));
		\eeq 
		\item[(ii)] $W(z)\le 11 n_0$,   
	\end{itemize}
	can be derived easily by \cite[Lemma 5.8]{R00}, hence we omit it.    If $W(z)<\frac{1}{2n_0}$, then
	$$
	W_{n_0}(z)\le W(z)<\frac{1}{2n_0} \ \Longrightarrow \ \bar W_{n_0}(z)= W_{\alpha_{n_0}}(z)+\frac{1}{8n_0} \ge W(z).
	$$
	Since $ \bar W_{n}(z)= W_{\alpha_n}(z)\le  W(z)$, this yields that the minimum is also reached for $n$.
	Thus Lemma \ref{L57} (i) and Lemma \ref{L56} imply the decrease condition \eqref{LMR58}  for any  $p\in  \partial_P\bar W (z)$.  
	It  remains   to show that  \eqref{LMR58} holds for any $p\in  \partial_P\bar W (z)$ also when  $z\in  W^{-1}\left(\left[\frac{1}{2n_0},11n_0\right]\right)$. 
	By Lemma  \ref{L56},   
	$$
	H\left(z,\p,p \right)= \inf_{u\in U}\left\{\langle p,f(z,u)\rangle+\p\, l(z,u)\right\}  \le  -\frac{\mathcal{W}(z)}{2}  \quad \forall p\in \partial_PW_{\alpha_{n_0}}(z)
	$$
	and, as a consequence (since $\p\ge0$ and $l\ge0$),   
	$$
	\inf_{u\in U}\left\{\langle p,f(z,u)\rangle\right\}<0 \quad \forall p\in \partial_PW_{\alpha_{n_0}}(z).
	$$
	Thus by Lemma \ref{L57} (iv) and \eqref{subd}, for any  $p\in  \partial_P\bar W (z)$ there is some $p_{n_0}\in  \partial_PW_{\alpha_{n_0}}(z)$ such that 
	$p=\Psi'_{n_0}(W_{\alpha_{n_0}}(z))\, p_{n_0}$ and
	$$
	\begin{array}{l}
	H \left(z,\frac{\p}{2},p \right)= \inf_{u\in U}\left\{\langle p,f(z,u)\rangle+\frac{\p}{2} l(z,u)\right\} \\ [1.5ex]
	\quad\qquad\qquad  =\inf_{u\in U}\left\{\Psi'_{n_0}(W_{\alpha_{n_0}}(z))\langle p_{n_0},f(z,u)\rangle+\frac{\p}{2} l(z,u)\right\}  \\ [1.5ex]
	\ds \quad\qquad\qquad \le \frac{1}{2}\, H \left(z,\p,p \right)\le -\frac{\mathcal{W}(z)}{4}.
	\end{array}
	$$
	
\end{proof}

This last lemma shows that the minimum in the definition of $\bar W(x)$ is always attained for $x\in\R^n\setminus\C$. Therefore, the function $\bar W$ is locally semiconcave  outside the target  (by Lemma \ref{InfC}, (v)).  On the other hand, $\bar W$ is continuous on $\overline{\R^n\setminus\C}$ because $0\le \bar W\le W$ and satisfies the decrease condition by \eqref{LMR58}, where, by Step 1, $ \frac{\mathcal{W}}{4}$ coincides with the composition of the positive, Lipschitz continuous,  and strictly increasing function  $\frac{\bar\gamma}{4}$ with the distance $\d$. Consequently, by Proposition \ref{Wweak} we can conclude that  $\bar W$ provides a $\ds\frac{\p}{2}$-MRF, which proves Theorem \ref{TLip}.
\qed
\subsection{Proof of Theorem \ref{CLipO}}\label{pCLipO}
Let us sketch how to adapt the arguments of the proof of Theorem \ref{TLip} to the case of a Lipschitz continuous $p_0$-OMRF, for which only the existence of a \emph{local} rate function $\gamma_\sigma$ is ensured. 

Let $W:\overline{R^n\setminus\C}\to[0,+\infty)$ be a locally Lipschitz continuous map, positive definite  and  proper on $\R^n\setminus\T$ and verifying
$$
H (x,\p,  \partial_LW(x))<0 \quad \forall x\in {{\R^n}\setminus\T}.     
$$
Our aim is  to prove the existence of  a $\ds\frac{\p}{2}$-OMRF $\bar W\le W$.
\vsm
{\sc Step 1.} By the Lipschitz continuity of $W$,  the set-valued map  $x\leadsto \partial_LW(x)$  has closed graph  with compact values, so that it  is  upper semicontinuous (see  \cite[Props. 4.3.3, 4.3.5]{Vinter} and \cite[Thm.1 and Cor. 1, pg. 41]{AC}).  At this point, one can derive that for any $\sigma>0$ there exists a positive, continuous, strictly increasing map $\gamma_\sigma$ defined in $(0,\sigma]$, such that  $W$ verifies
\bel{Oc0}
H  (x,\p,D_LW(x)) <-\gamma_\sigma(W(x)) \qquad \forall x\in W^{-1}((0,\sigma]),
\eeq
arguing exactly  as in  \cite[Prop. 3.1]{MR13}.   
If $\gamma_\sigma$ denotes  an arbitrary continuous and strictly increasing extension of $\gamma_\sigma$ to $(0,+\infty)$, by  Proposition \ref{W=d} it follows that for any continuous function $W_1 :\ol{\R^n\setminus \T}\to[0,+\infty)$, positive definite  and  proper on $\R^n\setminus\T$,  there exists a  continuous, strictly increasing function $\tilde\gamma_\sigma:(0,+\infty)\to(0,+\infty)$ such that $W$ also verifies
\bel{OPMRH} 
H (x,\p,  \partial_LW(x) )\le-\tilde\gamma_\sigma(W_1(x)) \qquad \forall x\in W^{-1}((0,\sigma])      
\eeq 
(and vice-versa, if $W$, $W_1$ satisfy \eqref{OPMRH},   then  \eqref{Oc0} holds true for some $\gamma_\sigma$ as above). Let us choose  $W_1=\mathbf d$. Similarly to Step 1 of the proof of Theorem  \ref{TLip},  we can suppose without loss of generality that  $\tilde\gamma_\sigma$ is   1-Lipschitz continuous and consider the map $\mathcal W_\sigma:=\tilde\gamma_\sigma \circ \mathbf d$, which is   $1$-Lipschitz continuous and positive definite  in $\R^n\setminus\T$.
%,  by Lemma \ref{InfC}, (iv).  
Therefore, recalling that $ \partial_PW(x)\subseteq  \partial_LW(x)$ for every $x$,  $W$ verifies
\bel{local}
H(x,p_0,\partial_PW(x))\leq -\mathcal W_\sigma(x)\qquad \forall x\in W^{-1}((0,\sigma]).
\eeq
\vsm
{\sc Step 2.} For any integer  $n\ge1$, let us  set  $\sigma_n:= 11n$. Hence $W$ verifies
\bel{localmn}
H(x,p_0,\partial_PW(x))\leq -\mathcal W_{\sigma_n}(x)\qquad \forall x\in W^{-1}((0,\sigma_n]),
\eeq
where it is easy to see that $(\mathcal W_{\sigma_n})_n$ is a decreasing sequence.   From now on, the proof proceeds similarly to Appendix \ref{A1}, with the crucial  differences  that the decrease rate $\mathcal W_{\sigma_n}$  in  \eqref{localmn} depends  on $\sigma_n$ and that the condition \eqref{localmn} is satisfied only in $W^{-1}((0,\sigma_n])$.  In particular, these facts imply that,  for any $n\ge1$, the inf-convolution $W_{\alpha_n}$ of $W$ depends  on $\mathcal W_{\sigma_n}$, since   $\alpha_n$ is given by
$$
\alpha_n:=\max\left\{8n(L^n_W)^2+1\,,\, \frac{2L^n_W(1+L^n_WL^n_f+\p L^n_l)}{m_n}+1\,,\, 11n\right\},
$$
where all the constants are the same as  in the proof of Theorem \ref{TLip}  and, in particular,    
$$ 
m_n =\min \left\{\mathcal W_{\sigma_n}(x):\quad x\in W^{-1}\left(\left[\frac{1}{2n},11n \right]\right)\right\}.
$$
Lemmas  \ref{L55}, \ref{L57}, dealing with the properties of  the approximations of $W$,  hold unchanged, while  Lemmas \ref{L56}, \ref{L58}  are now replaced by the following results.
\begin{Lemma}\label{L56O} Let $z\in  W^{-1}\left(\left[\frac{1}{2n},\sigma_n\right]\right)$ and $p\in\partial_PW_{\alpha_n}(z)$. Then
	\bel{LMR56} 
	H (z,\p, p )\le-\frac{\mathcal{W}_{\sigma_n}(z)}{2}.    
	\eeq 
\end{Lemma}
\begin{proof} The only delicate point in order to adapt  the proof of Lemma  \ref{L56} to the present setting,    is  that, given $z\in W^{-1}((0,\sigma_n])$, one has to apply the decrease condition in \eqref{localmn} not at $z$, but at the point $\bar y$ where the minimum in definition \eqref{Wal}  of  $W_{\alpha_n}(x)$ is obtained.  This can be done since $\bar y$ belongs to the sublevel set  $W^{-1}((0,\sigma_n])$ too; indeed,
	$$
	W(\bar y)= W_{\alpha_n}(z)-\alpha_n|\bar y-z|^2\le W_{\alpha_n}(z)\le W(z)\le\sigma_n.
	$$  
\end{proof}

\begin{Lemma}\label{L58O}  For all integer $n\ge 1$ and for all  $z\in  W^{-1}\left(\left[\frac{1}{n},\sigma_{n}-n\right]\right)$, one has  $\bar W(z)=
	\bar W_{n_0}(z)$ for some $n_0\in\{1,\dots, n\}$.
	Furthermore, if $p\in  \partial_P\bar W (z)$, then  
	\bel{LMR58local} 
	H \left(z,\frac{\p}{2},p \right)\le-\frac{\mathcal{W}_{\sigma_{n}}(z)}{4}.      
	\eeq 
\end{Lemma} 
\begin{proof} Going through  the proof of Lemma  \ref{L58}, the  crucial remark is that,  whenever the minimum 
	$$\bar W(z)=\bar W_{n_0}(z):=\min_{1\le k\le n}\bar W_n(z)$$ is obtained for some $n_0<n$, then $W(z)\le \sigma_{n_0}$. The last inequality implies that, when $W(z)\ge \frac{1}{2n_0}$, the point $z$ belongs  to the strip $W^{-1}\left(\left[\frac{1}{2n_0},\sigma_{n_0}\right]\right)$. Therefore,  arguing as  in the proof of Lemma  \ref{L58},  one can apply Lemma \ref{L56O} to derive that
	$$ 
	H \left(z,\frac{\p}{2},p \right)\le-\frac{\mathcal{W}_{\sigma_{n_0}}(z)}{4}.      
	$$
	Recalling that the sequence $(\mathcal{W}_{\sigma_{n}})$ is decreasing, this yields the decrease condition \eqref{LMR58local}.
	The proof in the case  $W(z)< \frac{1}{2n_0}$, where  one can assume $n_0=n$,  can be obtained  again by Lemma \ref{L56O} and the arguments of  the proof of Lemma  \ref{L58}. 
\end{proof} 
The decrease condition \eqref{MRHv} follows now by the arbitrariness of $n$ and, consequently, we have that $\bar W$ is a $\p/2$-OMRF.

\end{document}